%% file: Paths_2013-11-27-arxiv.tex
\documentclass[10pt,a4paper]{amsart}
\usepackage{amsthm,amsfonts,amsmath,amssymb,latexsym}
\usepackage{epsfig,graphics,color}
\usepackage[matrix,arrow,curve]{xy}
\usepackage{hyperref}
\usepackage{mathrsfs}
\usepackage{amsmath}
\usepackage{amsfonts}
\usepackage{amssymb}
\usepackage{graphicx}
\usepackage{dashbox}
\newtheorem{theorem}{Theorem}
\newtheorem{corollary}{Corollary}
\newtheorem{lemma}{Lemma}

\newtheorem{definition}{Definition}

\theoremstyle{definition}
\newtheorem{remark}{Remark}

\begin{document}
\title[Paths in complex projective space]{The space of paths in complex projective space with real boundary conditions}
\author{Nancy Hingston}
\address{Department of Mathematics and Statistics, College of New Jersey, Ewing, New
Jersey 08628, USA}
\email{hingston@tcnj.edu}
\author{Alexandru Oancea}
\address{Institut de Math\'ematiques de Jussieu -- Paris Rive Gauche, UMR 7586,
Universit\'e Pierre et Marie Curie \& CNRS, 4 place Jussieu, 75252 Paris, France.}
\email{oancea@math.jussieu.fr}
\date{November 27, 2013}

\begin{abstract}
We compute the integral homology of the space of paths in $\mathbb{C} P^{n}$
with endpoints in $\mathbb{R} P^{n}$, $n\ge1$ and its algebra structure with
respect to the Pontryagin-Chas-Sullivan product with $\mathbb{Z}/2$-coefficients.
\end{abstract}
\maketitle




\section*{Introduction}

Let $\mathcal{P}_{n}=\mathcal{P}_{\mathbb{R}P^{n}}\mathbb{C}P^{n}$, $n\geq1$
be the space of continuous paths $\gamma:[0,1]\rightarrow\mathbb{C}P^{n}$ with
endpoints $\gamma(0),\gamma(1)\in\mathbb{R}P^{n}$ and denote
\[
\mathbb{H}_{\cdot}(\mathcal{P}_{n}):=H_{\cdot+n}(\mathcal{P}_{n}
;\mathbb{Z}/2).
\]
More generally, given a $\mathbb{Z}$-graded module $V_{\cdot}$ and
$k\in\mathbb{Z}$ we denote $V[k]_{\cdot}$ or $V_{\cdot}[k]$ the $\mathbb{Z}
$-graded module $V[k]_{j}=V_{j}[k]:=V_{j+k}$. Thus $\mathbb{H}_{\cdot
}(\mathcal{P}_{n})=H_{\cdot}(\mathcal{P}_{n})[n]$.

We shall refer to $\mathbb{H}_{\cdot}(\mathcal{P}_{n})$ as the \emph{path
homology} of the pair $(\mathbb{C}P^{n},\mathbb{R}P^{n})$ and use for
readability the shorthand notation $\mathbb{H}$. This is a unital graded
algebra with respect to the \emph{Pontryagin-Chas-Sullivan product}
\[
\ast:\mathbb{H}_{i}\otimes\mathbb{H}_{j}\rightarrow\mathbb{H}_{i+j}.
\]
This product is heuristically defined by intersecting cycles
transversally over the evaluation maps at the endpoints, and then
concatenating. We refer to~\S\ref{sec:PCS} below for the precise definition. 
Our grading convention is such that $\mathbb{H}_{\cdot
}$ is supported in degrees $\geq-n$. The class of a point has degree $-n$ and
the unit has degree $0$, being represented by the cycle $[\mathbb{R}P^{n}]$ of
constant loops. 

The purpose of this paper is to compute the algebra
$(\mathbb{H}_{\cdot},*)$. The result is contained in
Theorem~\ref{thm:PCSproduct}. Along the way, we shall also compute
in~\S \ref{sec:homologyPn} the integral homology groups $H_{\cdot}
(\mathcal{P}_{n};\mathbb{Z})$.

\begin{theorem}
\label{thm:PCSproduct} The path homology algebra $(\mathbb{H}_{\cdot
}(\mathcal{P}_{n}),\ast)$, $n\ge1$ with $\mathbb{Z}/2$-coefficients admits the
following presentation:

\begin{itemize}
\item if $n\equiv1$ modulo $4$, then
\[
\mathbb{H}_{\cdot}(\mathcal{P}_{n})\simeq\langle H,S,Y \rangle/ \left\{
\begin{array}
[c]{cc}
[H,S]=1, \ [H,Y]=0,\ [S,Y]=H^{n-1}Y^{2}, & \\
S^{2}=0,\ H^{n+1}=0 &
\end{array}
\right\}
\]
with $\langle H,S,Y\rangle$ being the free graded unital algebra over $\mathbb{Z}/2$
with generators $H,S,Y$ of degrees
\[
|H|=-1,\qquad|S|=1,\qquad|Y|=n.
\]

\item if $n\equiv3$ modulo $4$, then
\[
\mathbb{H}_{\cdot}(\mathcal{P}_{n})\simeq\langle H,S,Y \rangle/ \left\{
\begin{array}
[c]{cc}
[H,S]=1, \ [H,Y]=0,\ [S,Y]=0, & \\
S^{2}=0,\ H^{n+1}=0 &
\end{array}
\right\}
\]
where $\langle H,S,Y\rangle$ has the same meaning as above.

\item if $n$ is even
\[
\mathbb{H}_{\cdot}(\mathcal{P}_{n})\simeq\langle H,T,Y\rangle/ \left\{
\begin{array}
[c]{cc}
[H,T]=H, \ [H,Y]=0,\ [T,Y]=Y, & \\
T^{2}=T,\ H^{n+1}=0 &
\end{array}
\right\}
\]
with $\langle H,T,Y\rangle$ being the free graded unital algebra over $\mathbb{Z}/2$
with generators $H,T,Y$ of degrees
\[
|H|=-1, \qquad|T|=0, \qquad|Y|=n.
\]

\end{itemize}
\end{theorem}

As far as we know, this is the first computation both of the homology of a
space of paths with endpoints on a non-trivial submanifold, and of the product
structure that it carries. Note that the resulting algebra is non-commutative. This is a generic expectation for Pontryagin-Chas-Sullivan algebras of path spaces - in contrast to Chas-Sullivan algebras for loop spaces (see also the section below on Motivations). We find it remarkable that the ring that we compute is as non-commutative as it is, especially given that the fundamental groups are commutative. The most surprising relation is perhaps $SY+YS=H^{n-1}Y^2$ if $n\equiv 1 (\mathrm{mod}\ 4)$, which shows that the induced product on level homology for the standard norm functional described in~\S\ref{sec:concatenation} is not commutative either. Here by ``level homology" we mean the graded object associated to the filtration of $\mathcal P_n$ by sublevel sets of the norm functional (see also~\cite[\S1.2]{Goresky-Hingston}). 

\subsection{Motivations and further questions.}

We shall restrict ourselves to a brief discussion, although our motivations
are manifold. They stem from symplectic topology and from string topology.

Let $M$ be a closed manifold and $N$ a closed submanifold. The homology
(with local coefficients) of the space $\mathcal{P}_{N} M$ coincides with the
wrapped Floer homology of the conormal bundle $T^{*}_{N}M:=\{(x,p)\in T^{*}M\,
: \, p|_{T_{x}N}=0\}$~\cite{Abbondandolo-Portaluri-Schwarz,Abouzaid-Seidel}.
(This is an exact Lagrangian submanifold in $T^{*}M$.) Although the
Pontryagin-Chas-Sullivan product that we define in this paper has not been
studied outside the fundamental case $N=pt$, this isomorphism is expected to
intertwine it with the (half-)pair-of-pants product on wrapped Floer homology,
in the spirit of~\cite{AS2}. Our computation can be understood as a
computation of the wrapped Floer homology ring of $T^*_{\mathbb{R}P^n}\mathbb{C}P^n$.

This correspondence in turn yields many questions. It is known that the
wrapped Fukaya category of the cotangent bundle $T^{*}M$ is generated by
either the zero section or by a fiber, and hence its Hochschild homology is
the homology of the free loop space~\cite{Abouzaid2011a}. It is
reasonable to expect that the Fukaya category is generated by the conormal
bundle of any given submanifold. Tobias Ekholm asked the following question: is
the Hochschild homology of the algebra $(\mathbb{H}_{\cdot},*)$ equal to the
homology of the free loop space on $\mathbb{C}P^{n}$? This is a formality
question. What is the $A_{\infty}$-algebra structure on chains on
$\mathcal{P}_{n}$? This should be tractable, since our description of
$\mathcal{P}_{n}$ via Morse theory is essentially complete (and, as Thomas Kragh
pointed out, it should even yield the homotopy theory of $\mathcal{P}_{n}$). What is the Pontryagin-Chas-Sullivan ring   $H_\cdot(\mathcal P_{\mathbb{R}P^k}\mathbb{C}P^n)$, $k<n$? 
What
other general tools for computing the algebra structure on $H_{\cdot
}(\mathcal{P}_{N}M)$ does one have? A specific question here is to establish the relative analogue of the multiplicative spectral sequence in~\cite{Cohen-Jones-Yan}. 
Can one compute the homology or the algebra structure on  $H_{\cdot}(\mathcal{P}_{N}M)$ from minimal model data on the pair $(M,N)$? What about the module structure of $H_{\cdot}(\mathcal{P}_{N}M)$ over the Chas-Sullivan ring $H_\cdot(\mathcal L M)$, where $\mathcal LM$ is the free loop space of $M$?

Another line of motivation stems from ``quantum string topology'', which is an
ongoing project of the second author~\cite{Oancea-Montreal}. From this
perspective, the key point is that $\mathbb{R}P^{n}$ is a monotone Lagrangian
inside $\mathbb{C}P^{n}$, and the Pontryagin-Chas-Sullivan product that we
study in this paper can be further deformed using holomorphic discs with
boundary on $\mathbb{R}P^{n}$.

More generally, we believe that it will be important to incorporate path
spaces $\mathcal{P}_{N} M$ into string topology in a systematic way. The only
attempt in this direction that we are aware of
is~\cite{Blumberg-Cohen-Teleman}.

Besides all the above, the thing that kept us going to the end of the
computations was the sheer beauty of the geometry and of the structures that emerged.

\subsection{The Pontryagin-Chas-Sullivan product}
\label{sec:PCS} Our definition of the product is a variation of the 
definitions implemented in~\cite{Chataur,Cohen-Jones} for the Chas-Sullivan product. 
We place
ourselves in the context of paths of Sobolev class $W^{1,2}$, so that
$\mathcal{P}_{n}$ is a Hilbert manifold. This does not result in any loss of
generality since the groups $\mathbb{H}_{i}$ are invariant under strengthening
the regularity scale of the paths under consideration. Let $\mathrm{ev}
_{t}:\mathcal{P}_{n}\rightarrow\mathbb{R}P^{n}$, $t=0,1$ be the evaluation
maps at the endpoints given by $\mathrm{ev}_{t}(\gamma):=\gamma(t)$. The maps
$\mathrm{ev}_{t}$, $t=0,1$ are submersions and the fiber product
\[
\mathcal{C}_{n}:=\mathcal{P}_{n}\ {_{\mathrm{ev}_{1}}}\!\!\times
_{\mathrm{ev}_{0}}\mathcal{P}_{n}:=\{(\gamma,\delta)\in\mathcal{P}_{n}
\times\mathcal{P}_{n}\ :\ \mathrm{ev}_{1}(\gamma)=\mathrm{ev}_{0}(\delta)\}
\]
is a codimension $n$ Hilbert submanifold of $\mathcal{P}_{n}\times
\mathcal{P}_{n}$. We denote the inclusion
\[
s:\mathcal{C}_{n}\hookrightarrow\mathcal{P}_{n}\times\mathcal{P}_{n}.
\]
We also consider the \emph{concatenation map (at time $t=1/2$)}
\[
c:\mathcal{C}_{n}\hookrightarrow\mathcal{P}_{n}
\]
defined by $c(\gamma,\delta)(t):=\gamma(2t)$ for $0\leq t\leq1/2$ and
$c(\gamma,\delta)(t):=\delta(2t-1)$ for $1/2\leq t\leq1$. Note that $c$ is a
codimension $n$ embedding and its image is the set $\{\gamma\in\mathcal{P}
_{n}\,:\,\gamma(\frac{1}{2})\in\mathbb{R}P^{n}\}$. We thus have a diagram
\[
\mathcal{P}_{n}\overset{c}{\longleftarrow}\mathcal{C}_{n}\overset
{s}{\longrightarrow}\mathcal{P}_{n}\times\mathcal{P}_{n}
\]
Restricting to $\mathbb{Z}/2$-coefficients, we obtain the following diagram in
homology
\[
\mathbb{H}_{i+j}\overset{c_{\ast}}{\longleftarrow}H_{i+j+n}(\mathcal{C}
_{n})\overset{s_{!}}{\longleftarrow}H_{i+j+2n}(\mathcal{P}_{n}\times
\mathcal{P}_{n})\overset{AW}{\longleftarrow}\mathbb{H}_{i}\otimes
\mathbb{H}_{j}.
\]
Here $AW$ denotes the Alexander-Whitney shuffle product~\cite[p.~268]
{Greenberg-Harper} and the map $c_{*}$ is the map induced in homology by
concatenation. As for the \emph{shriek map} $s_{!}$, this is defined as the
composition
\[
H_{\cdot}(\mathcal{P}_{n}\times\mathcal{P}_{n})\to H_{\cdot}(\mathcal{P}
_{n}\times\mathcal{P}_{n},\mathcal{P}_{n}\times\mathcal{P}_{n}\setminus
\mathcal{C}_{n})\to H_{\cdot}(\nu(\mathcal{C}_{n}),\nu(\mathcal{C}
_{n})\setminus\mathcal{C}_{n})\overset{\tau^{-1}}{\longrightarrow}H_{\cdot
-n}(\mathcal{C}_{n}).
\]
Here the first map is induced by inclusion, the second map is given by
excision after having identified the normal bundle $\nu(\mathcal{C}_{n})$ to
$\mathcal{C}_{n}$ inside $\mathcal{P}_{n}\times\mathcal{P}_{n}$ to a tubular
neighborhood of $\mathcal{C}_{n}$, and $\tau^{-1}$ is the (inverse of) the
Thom isomorphism.

The Pontryagin-Chas-Sullivan product is defined to be
\[
\ast:\mathbb{H}_{\cdot}\otimes\mathbb{H}_{\cdot}\rightarrow\mathbb{H}_{\cdot
},\qquad\alpha\ast\beta:=c_{\ast}s_{!}AW(\alpha\otimes\beta).
\]

\begin{remark}
[On the Chas-Sullivan and Pontryagin products]This definition holds more
generally for the space $\mathcal{P}_{N} M$ of paths inside a manifold $M$
with endpoints on a submanifold $N$. As such, it generalizes the definition of
the Chas-Sullivan product: by ``folding up" at time $1/2$, free loops inside a
manifold $P$ can be viewed as paths inside $P\times P$ with endpoints on the
diagonal $\Delta\subset P\times P$, and the Chas-Sullivan product can be
expressed in terms of the Pontryagin-Chas-Sullivan product for $\mathcal{P}
_{\Delta}(P\times P)$.

This definition also generalizes that of the Pontryagin product on the based
loop space $\Omega M=\mathcal{P}_{pt} M$. This is our motivation for the
choice of terminology ``Pontryagin-Chas-Sullivan product".
\end{remark}

\begin{remark}
[On orientability and integral coefficients]The key ingredient in the
definition of the product is the shriek map $s_{!}$, which in turn relies on
the Thom isomorphism
\[
H_{\cdot}(\nu(\mathcal{C}_{n}),\nu(\mathcal{C}_{n})\setminus\mathcal{C}
_{n})\simeq H_{\cdot-n}(\mathcal{C}_{n}).
\]
The isomorphism holds as such with $\mathbb{Z}/2$-coefficients, and it holds
with $\mathbb{Z}$-coefficients if the normal bundle $\nu(\mathcal{C}_{n})$ is
orientable. In case the normal bundle is not orientable, the isomorphism only
holds under the form
\[
H_{\cdot}(\nu(\mathcal{C}_{n}),\nu(\mathcal{C}_{n})\setminus\mathcal{C}
_{n})\simeq H_{\cdot-n}(\mathcal{C}_{n};o_{\nu(\mathcal{C}_{n})}),
\]
where $o_{\nu(\mathcal{C}_{n})}$ is the local system of orientations for
$\nu(\mathcal{C}_{n})$. In our situation $\mathcal{C}_{n}=(\mathrm{ev}
_{1},\mathrm{ev}_{0})^{-1}(\Delta)$, with $(\mathrm{ev}_{1},\mathrm{ev}
_{0}):\mathcal{P}_{n}\times\mathcal{P}_{n}\to\mathbb{R}P^{n}\times
\mathbb{R}P^{n}$, and we have
\[
\nu(\mathcal{C}_{n})=(\mathrm{ev}_{1},\mathrm{ev}_{0})^{*}\nu_{\mathbb{R}
P^{n}\times\mathbb{R}P^{n}}\Delta.
\]
Using that $\nu_{\mathbb{R}P^{n}\times\mathbb{R}P^{n}}\Delta\simeq
T\mathbb{R}P^{n}$ we obtain that $\nu(\mathcal{C}_{n})$ is orientable if and
only if $n$ is odd. Moreover, a straightforward computation shows that the
Pontryagin-Chas-Sullivan product is defined with integral coefficients on $H_{\cdot
}(\mathcal{P}_{n}; o_{\mathrm{ev}_{0}^{*}T\mathbb{R}P^{n}})$. For the more
general case of the space $\mathcal{P}_{N}M$, the Pontryagin-Chas-Sullivan
product is defined with integral coefficients on $H_{\cdot}(\mathcal{P}_{N}M;
o_{\mathrm{ev}_{0}^{*}TN})$.
\end{remark}

\subsection{Convention for concatenation.}

\label{sec:concatenation}

In the above discussion we have used the concatenation map $c:\mathcal{C}
_{n}\rightarrow\mathcal{P}_{n}$ at time $t=1/2$. This is of course not
associative. However, within the setup of paths of Sobolev class $W^{1,2}$
there is a very elegant way to obtain associativity by using what we will call
the \emph{energy minimizing concatenation map}
\[
c_{\min}:\mathcal{C}_{n}\rightarrow\mathcal{P}_{n}.
\]
This map appeared for the first time in~\cite[Lemma~2.4, see also~\S 10.6]
{Goresky-Hingston} and is closely related to the classical \emph{energy
functional}
\[
E:\mathcal{P}_{n}\rightarrow\mathbb{R}_{+},\qquad E(\gamma):=\int_{0}
^{1}|\gamma^{\prime}|^{2},
\]
and even more to what we call the \emph{($L^{2}$-) norm} functional
\[
F:\mathcal{P}_{n}\rightarrow\mathbb{R}_{+},\qquad F(\gamma):=\left(  \int
_{0}^{1}|\gamma^{\prime}|^{2}\right)  ^{1/2}=\Vert\gamma^{\prime}\Vert_{L^{2}
}.
\]
The map $c_{\min}$ is defined by
\[
c_{\min}(\gamma,\delta)(t):=\left\{
\begin{array}
[c]{c}
\gamma(\frac{t}{s})\text{ \ if }0\leq t\leq s\\
\delta(\frac{t-s}{1-s})\text{ if }s\leq t\leq1
\end{array}
\right\}
\]
where
\[
s=s_{\min}:=\frac{F(\gamma)}{F(\gamma)+F(\delta)}.
\]
The key point is that the concatenation product $c_{\min}$ produces the unique
piecewise linear concatenation of minimum energy. This implies in turn
associativity
\[
c_{\min}(c_{\min}(\alpha,\beta),\gamma)=c_{\min}(\alpha,c_{\min}(\beta
,\gamma)).
\]
This can of course also be checked directly, and we do encourage the reader to
perform this surprising calculation. The map $c_{\min}$ is homotopic to the
concatenation $c$ and the Pontryagin-Chas-Sullivan product can therefore
alternatively be defined using $c_{\min}$.

The map $c_{\min}$ and the norm $F$ satisfy the following remarkable
and useful identity:
\[
F(c_{\min}(\gamma,\delta))=F(\gamma)+F(\delta).
\]
This is the key to Lemma~\ref{lem:Crit} in~\S \ref{sec:gen}, which relates the
Pontryagin-Chas-Sullivan product and min-max critical values of the functional
$F$. We would like to emphasize that, from the joint point of view of
variational calculus and of product structures on $H_{\cdot}(\mathcal{P}_{n}
)$, the norm $F$ plays a more fundamental role than the energy $E$, although
they have the same Morse theory (same critical points, index, and nullity).

\medskip 

\noindent\textbf{Convention.} We shall deal in the sequel with concatenations
of arbitrary many paths, and we shall always perform the concatenation using
the map $c_{\min}$.

\subsection{Structure of the paper.}

The paper contains three sections whose names are self-explanatory: Geometry,
Homology, and Product. Our main tool is Morse theory for the norm functional
$F$ for the Fubini-Study metric on $\mathbb{C}P^{n}$, and the key point is the
existence of completing manifolds in the sense of
Definition~\ref{def:completing}.

\noindent\textbf{Acknowledgements.} Both authors would like to acknowledge the
hospitality of and inspiring working conditions at the Institute for Advanced
Study in Princeton during the Summer 2013, when the main work was done. A.O.
was partially supported by ERC Starting Grant StG-259118-Stein.


\section{The geometry of $\mathcal{P}_{n}$}

In all subsequent sections we work with paths of Sobolev class $W^{1,2}$ on
the space $\mathbb{C}P^{n}$ equipped with the Fubini-Study metric. A
traditional approach to path spaces (\cite{Bott-Lectures_Morse_theory},
\cite{Milnor-Morse_theory}) that works well with the finite dimensional
approximation of Morse is the Morse theory for the \emph{energy} functional
\[
E:\mathcal{P}_{n}\rightarrow\mathbb{R},\text{ \ \ \ \ }E(\gamma):=\int_{0}
^{1}|\gamma^{\prime}(t)|^{2}dt.
\]
Following~\cite{Goresky-Hingston} we will use instead the ($L^{2}$-)\textit{
norm}
\[
F:\mathcal{P}_{n}\rightarrow\mathbb{R},\text{ \ \ \ \ }F(\gamma):=\sqrt
{E(\gamma)}=\left(  \int_{0}^{1}|\gamma^{\prime}(t)|^{2}dt\right)  ^{\frac
{1}{2}}=\|\gamma^{\prime}\|_{L^{2}}.
\]
Note that
\[
F(\gamma)\geq length(\gamma),
\]
with equality if and only if $\gamma$ is parameterized proportional to
arclength; it is a good approximation to the truth to think of $F$ as the
length. The norm $F$ is not differentiable on the constant paths; we consider
them critical points of $F$. The functions $E$ and $F$ have (by a simple
computation) the same critical points, with the same index and nullity and
thus produce the same Morse theory. However the norm $F$ behaves better with
respect to concatenation of paths, as explained in~\S \ref{sec:concatenation}.
The path homology algebra of the pair $(\mathbb{C}P^{n},\mathbb{R}P^{n})$ is
of course independent of the metric; it is standard procedure to use Morse
theory with a special function having simple critical points to compute
topological data.

The critical points of $F$ are the geodesics that are perpendicular to
$\mathbb{R}P^{n}$ at both ends. Geodesics of the Fubini-Study metric have the
following simple description~\cite[Proposition~3.32]{Besse}: given a point $z\in\mathbb{C}P^{n}$
and a unit tangent vector $v\in T_{z}$ $\mathbb{C}P^{n}$, the geodesic
$\gamma_{z,v}(s):=\exp_{z}(sv)$, $s\in\mathbb{R}$ is periodic of period $\pi$
and is a great circle parametrized by arc-length on the unique complex line
$\mathcal{\ell}_{z,v}$ ($\simeq\mathbb{C}P^{1}$) passing through $z$ and
tangent to $v$. Note that the Fubini-Study metric on $\mathbb{C}P^{n}$ is such
that complex lines are round spheres of curvature $4$ and circumference $\pi$.
In particular, if $x\in\mathbb{R}P^{n}$ and $v\in SN\mathbb{R}P^{n}$ (the unit
normal bundle of $\mathbb{R}P^{n}$ in $\mathbb{C}P^{n}$), so that the geodesic
$\gamma_{x,v}$ starts on $\mathbb{R}P^{n}$ in an orthogonal direction, it will
meet (orthogonally) again $\mathbb{R}P^{n}$ at times $s=k\pi/2$, $k\geq1$. If
$k$ is even then $\gamma_{x,v}(k\pi/2)=x$ and $\gamma_{x,v}^{\prime}
(k\pi/2)=v$, whereas if $k$ is odd then $\gamma_{x,v}(k\pi/2)=x^{\ast}$, the
antipode of $x$ in $\mathcal{\ell}_{x,v}$, and $\gamma_{x,v}^{\prime}
(k\pi/2)=v^{\ast}$, where $v^{\ast}$ is the image of $v$ under the antipodal
map on $\mathcal{\ell}_{x,v}$. The point $x^{\ast}$ can be alternatively
described as the cutpoint of $x$ inside $\mathbb{R}P^{n}$ in the direction
$-Iv\in T_{x}\mathbb{R}P^{n}$, where $I:T_{x}\mathbb{C} P^{n}\rightarrow
T_{x}\mathbb{C}P^{n}$ is the complex structure. Yet another description of
$x^{\ast}$ is the following: the complex line $\mathcal{\ell}_{x,v}$
intersects $\mathbb{R}P^{n}$ along its equator $\mathbb{R} \mathcal{\ell
}_{x,v}$, which is the unique real projective line in $\mathbb{R}P^{n}$
passing through $x$ in the direction $Iv$. The geodesic $\gamma_{x,-Iv}$ is a
parametrization of $\mathbb{R}\mathcal{\ell}_{x,v}$ by arc-length and the
cutpoint is $\gamma_{x,-Iv}(\pi/2)=x^{\ast}$. In fact, the cut-locus of $x$ in
$\mathbb{R}P^{n}$ is a real hyperplane whose intersection with $\mathbb{R}
\mathcal{\ell}_{x,v}$ is the point $x^{\ast}$. 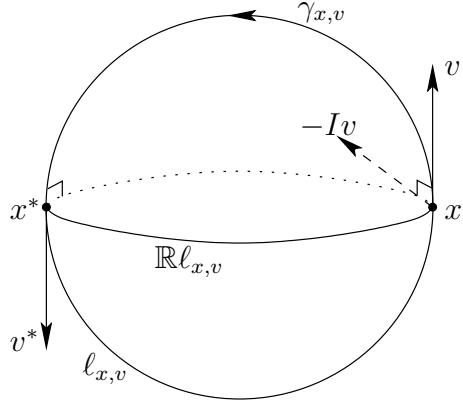
\begin{figure}[ptb]
\begin{center}
\input{geodesics.pstex_t}
\end{center}
\caption{Geodesics in $\mathbb{C}P^{n}$.}
\label{fig:geodesics}
\end{figure}

It follows from this discussion that the critical set of $F$ is a disjoint
union of manifolds $K_{k}$, $k\geq0$ that consist respectively of geodesics of
length $k\pi/2$ starting (and ending) on $\mathbb{R}P^{n}$ perpendicularly. If
$k=0$ then $K_{0}=\mathbb{R}P^{n}$, the space of constant paths. If $k\geq1$
there is a natural identification of $K_{k}$ with $SN\mathbb{R}P^{n}$ given
by
\[
\varphi_{k}:SN\mathbb{R}P^{n}\rightarrow K_{k},\text{ \ \ \ \ }\varphi
_{k}(x,v):=\gamma_{x,v}|_{\left[  0,k\pi/2\right]  }.
\]
Of course, the unit normal bundle $SN\mathbb{R}P^{n}$ is isomorphic to the
unit tangent bundle of $\mathbb{R}P^{n}$, denoted $ST\mathbb{R}P^{n}$, via
\[
SN\mathbb{R}P^{n}\overset{\simeq}{\longrightarrow}ST\mathbb{R}P^{n}
,\qquad(x,v)\mapsto(x,-Iv).
\]

Recall that, given a function $f:X\rightarrow\mathbb{R}$ of class $C^{2}$ on a
Hilbert manifold $X$ and a submanifold $K$ consisting of critical points of
$f$, we say that $K$ is a \textit{Morse-Bott critical manifold of index
}$\iota(K)$ \textit{and nullity} $\eta(K)$ if at each point $p\in K$ the
second derivative $H_{f}(p)$ is nondegenerate on the normal bundle to $K$ in
$X$, and if $H_{f}(p)$ has index $\iota(K)$ and nullity $\eta(K)$. (By the nullity $\eta(K)$ we mean the dimension of a maximal null subspace of the tangent space to $X$ at a point in $K$.) Note that the first condition is equivalent to requiring the tangent space $T_{p}K$ to
coincide with the null-space of $H_{f}(p)$.

\begin{lemma}
\label{lem:index+nullity} The manifolds $K_{k}$, $k\geq0$ are Morse-Bott.
Their index and nullity are respectively given by
\[
\iota(K_{0})=0,\text{ \ \ \ \ }\eta(K_{0})=n
\]
and, if $k\geq1$,
\[
\iota(K_{k})=1+(k-1)n,\text{ \ \ \ \ }\eta(K_{k})=2n-1.
\]
\end{lemma}

The proof is based on a space of ``half-circles'' in $\mathbb{C}P^{n}$ that captures the Morse theory of the norm functional on the space $\mathcal{P}_{n}$.

\medskip 

\textbf{Circles and Half-circles in }$\mathbb{C}P^{n}$

We begin with several descriptions of the space of ``vertical circles'' in
$\mathbb{C}P^{n}$. Given $x\in\mathbb{R}P^{n}$ and $v\in SN\mathbb{R}P^{n}$,
non-constant and unparametrized circles on the complex line $\mathcal{\ell
}_{x,v}$ form a well-defined conformally invariant class. (The authors could
not resist pointing out that circles on $\mathbb{C}P^{1}$ (including constant
ones) are precisely the images of great circles on the round $3$-sphere via
the Hopf map. Similarly, circles on $\mathbb{C}P^{n}$, defined as circles on
the complex lines $\ell_{p,v}$, are precisely the images of great circles on
the round $2n+1$-sphere via the Hopf map.) As such, if $x^{\prime}
\in\mathbb{R} \mathcal{\ell}_{x,v}$, $x^{\prime}\neq x$, there is a unique
circle through $x$, tangent to $v$ and passing through $x^{\prime}$. Note this
circle is orthogonal to $\mathbb{R} \mathcal{\ell}_{x,v}$ at $x$ and at
$x^{\prime}$ (thus "vertical"), and can also be found as follows: if
$x^{\prime},x\in\mathbb{R}P^{n}$ with $x^{\prime}\neq x$, then there is a
unique complex line containing them, and a unique circle containing $x$ and
$x^{\prime}$ and meeting $\mathbb{R}P^{n}$ orthogonally, the unique
\textit{vertical} circle containing $x$ and $x^{\prime}$.

Yet another description of the vertical circle containing $x$ and $x^{\prime}
$: embed $\ell_{x,v}$ isometrically inside $\mathbb{R}^{3}= \mathbb{R}
^{2}\times\mathbb{R}$ as the sphere of radius $1/2$ with equator
$\mathbb{R}\ell_{x,v}=\ell_{x,v}\cap(\mathbb{R}^{2}\times\{0\})$ and
$v\in\{0\}\times\mathbb{R}^{+}$. Then $x$ and $x^{\prime}$ lie on the equator,
and the vertical circle is the intersection of the sphere $\ell_{x,v}$ with
the unique vertical plane containing $x$ and $x^{\prime}$.

For fixed $x$ and $v$, $\mathbb{R}\ell_{x,v}$ is naturally identified with
$S_{\pi}^{1}:=\mathbb{R}/\pi\mathbb{Z}$ via
\[
x^{\prime}=\exp_{x}(-\theta Iv),
\]
$\theta\in S_{\pi}^{1}$ . (In particular, the antipode $x^{\ast}$ of $x$ on
$\ell_{x,v}$ corresponds to $\theta=\frac{\pi}{2}$. The sign of $\theta$ is
chosen so that the equator $\{x^{\prime}\}$ is parameterized counterclockwise
from above.) The principal $S_{\pi}^{1}$-bundle over $SN\mathbb{R}P^{n}$ with
fiber $\mathbb{R}\ell_{x,v}$ is therefore identified with the product
\[
\{(x,v,\theta)\}=SN\mathbb{R}P^{n}\times S_{\pi}^{1}.
\]
The parameters $x,v,\theta$ also determine a unique \textit{vertical
half-circle}, denoted $C_{x,v,\theta}$; the half circle $C_{x,v,\theta}$ is
defined to be the intersection of the \textit{upper} hemisphere in $\ell
_{x,v}$ (as just embedded in $\mathbb{R}^{3}$) with the unique vertical plane
containing $x$ and $x^{\prime}$. We parametrize the half-circle $C_{x,v,\theta
}$ on the interval $[0,1]$ with constant speed so that it begins at $x$ and
ends at $x^{\prime}$. If $x^{\prime}=x$ we define $C_{x,v,\theta}$ to be the
constant circle at $x$ and we think of it as having the tangent vector $v$
attached to it. Thus we have a map
\[
C:SN\mathbb{R}P^{n}\times S_{\pi}^{1}\rightarrow\mathcal{P}_{n},\qquad
(x,v,\theta)\mapsto C_{x,v,\theta}.
\]
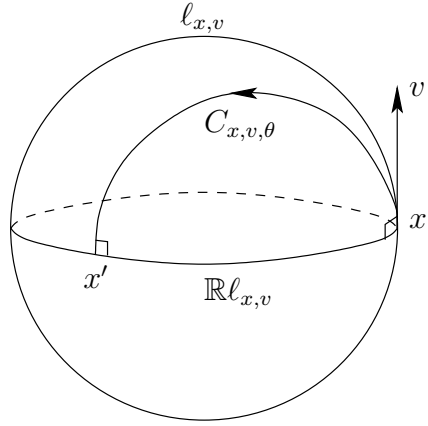
\begin{figure}[ptb]
\begin{center}
\input{half-circles.pstex_t}
\end{center}
\caption{Vertical half-circle in $\mathbb{C}P^{n}$.}
\label{fig:half-circles}
\end{figure}

We define
\[
Y_{1}:=SN\mathbb{R}P^{n}\times S_{\pi}^{1}=\{(x,v,\theta)\},
\]
which we think of as being the space of vertical half-circles as above. We
have smooth evaluation maps $\mathrm{ev}_{0}$, $\mathrm{ev}_{1}:Y_{1}
\rightarrow\mathbb{R}P^{n}$ defined by
\[
\mathrm{ev}_{0}(p,v,\theta):=C_{x,v,\theta}(0)=x,\qquad\mathrm{ev}
_{1}(x,v,\theta):=C_{x,v,\theta}(1)=x^{\prime}.
\]
Both evaluation maps are submersions. We define
\[
Y_{k}:={Y_{1}}\ _{\mathrm{ev}_{1}}\!\!\times_{\mathrm{ev}_{0}}{Y_{1}
}\ _{\mathrm{ev}_{1}}\!\!\times\dots\times_{\mathrm{ev}_{0}}{Y_{1}},
\]
where the number of factors in the fiber product is equal to $k\geq1$. We
think of this as being the space of paths $[0,k]\rightarrow\mathbb{C}P^{n}$
whose restriction to each interval $[j,j+1]$, $j\in\{0,\dots,k-1\}$ is a
vertical half-circle $C_{x_{j},v_{j},\theta_{j}}$. These half-circles have
matching endpoints and, if any of them is constant, it also has a unit normal
vector $v_{j}\in SN\mathbb{R}P^{n}$ attached to it. We denote $x_{k}$ the
endpoint of $C_{x_{k-1},v_{k-1},\theta_{k-1}}$. Note that two distinct points
$x,y\in\mathbb{R}P^{n}$ determine a unique complex line $\ell_{x,y}
\subset\mathbb{C}P^{n}$, a unique vertical circle on $\ell_{x,y}$ passing
through $x$ and $y$, and thus \emph{two} (unparametrized) vertical
half-circles with endpoints $x$ and $y$. As a consequence, the manifold
$Y_{k}$ is naturally parametrized near an element with $x_{j}$ is distinct
from $x_{j+1}$ for all $j\in\{0,\dots k-1\}$ by the sequence $(x_{0}
,x_{1},\dots,x_{k})\in(\mathbb{R}P^{n})^{k+1}$ In particular the open subset
of paths in $Y_{k}$ with adjacent $x_{j}$ distinct \textit{embeds} in
$\mathcal{P}_{n}$ and is locally diffeomorphic to an open set in
$(\mathbb{R}P^{n})^{k+1}$ . This is consistent with the fact that
\[
\dim\,Y_{k}=(k+1)n.
\]

The spaces $Y_{k}$, $k\ge1$ have the following important features: \renewcommand{\theenumi}{\roman{enumi}}

\begin{enumerate}
\item they naturally ``evaluate'' into $\mathcal{P}_{n}$ via the map
\begin{equation}
\label{eq:varphik}\varphi_{k}:Y_{k}\to\mathcal{P}_{n}
\end{equation}
defined by concatenating the vertical half-circles that constitute an element
of $Y_{k}$. More precisely, we have
\[
\varphi_{k}\big( (x_{0},v_{0},\theta_{0}),\dots,(x_{k-1},v_{k-1},\theta
_{k-1})\big):= C_{x_{0},v_{0},\theta_{0}}\cdot C_{x_{1},v_{1},\theta_{1}}
\cdot\ldots\cdot C_{x_{k-1},v_{k-1},\theta_{k-1}}.
\]

\item the critical set $K_{k}$ naturally embeds into $Y_{k}$ with codimension
$1+(k-1)n=\iota(K_{k})$ via the map
\begin{equation}
\label{eq:KkYk}\gamma_{x,v}|_{[0,k\pi/2]}\mapsto\big((x,v,\frac{\pi}
{2}),(x^{\prime},v^{\prime},\frac{\pi}{2}),(x,v,\frac{\pi}{2}),\dots\big).
\end{equation}
where $(x^{\prime},v^{\prime})$ is the image of $(x,v)$ under the derivative
of the antipodal map on $\mathcal{\ell}_{x,v}$. We denote
\[
L_{k}\subset Y_{k}
\]
the image of this embedding. Note that $\varphi_{k}$ is itself an embedding
near $L_{k}$.

\item the image of $Y_{k}$ in $\mathcal{P}_{n}$ contains piecewise geodesics
that are not smooth. These paths are not critical points of $F$, but the value
of $F$ at these points is the maximum value $k\pi/2$. However, after a small
perturbation of the map $\varphi_{k}$ in the direction $-\nabla F$, the image
of $Y_{k}\setminus L_{k}$ will be contained in $\mathcal{P}_{n}^{<k\pi/2}$.
\end{enumerate}

We will see in~\S \ref{sec:homologyPn} that, due to the above properties, the
spaces of vertical half-circles $Y_{k}$ faithfully reflect the Morse theory of
$F$ and actually carry the topology of $\mathcal{P}_{n}$. The relevant notion
is that of a \emph{completing manifold} (see Definition~\ref{def:completing} below).

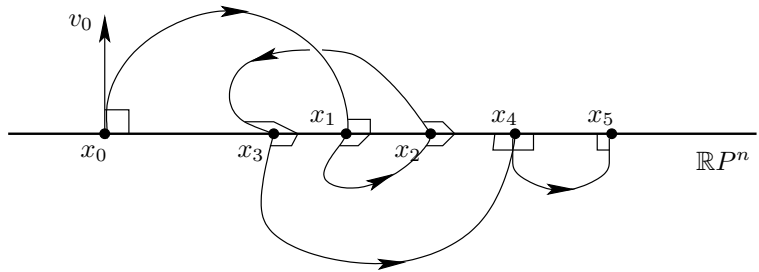
\begin{figure}[ptb]
\begin{center}
\input{Yk.pstex_t}
\end{center}
\caption{An element of $Y_{5}$.}
\label{fig:Yk}
\end{figure}

\begin{proof}
[Proof of Lemma~\ref{lem:index+nullity}]\bigskip The statement concerning
$K_{0}$ is a general fact~\cite[Proposition~2.4.6]{Klingenberg}. We now focus
on the case $k\geq1$.

That $\eta(K_{k})=2n-1$ follows from an explicit computation of which we only
give a sketch and omit the details. The null-space of $E$ (or $F$) at a
critical point $\gamma\in\mathcal{P}$ is spanned by Jacobi fields along
$\gamma$ whose covariant derivative at the endpoints is orthogonal to
$\mathbb{R}P^{n}$. (The explicit formula for the second derivative of $E$ can
be found in~\cite[p.~663]{Hingston-Kalish} and it implies this condition in
view of the fact that $\mathbb{R}P^{n}$ is totally geodesic inside
$\mathbb{C}P^{n}$.) Since one has explicit formulas for the Jacobi fields
along geodesics for the Fubini-Study metric on $\mathbb{C}P^{n}$ (see for
example~\cite[Proposition~3.34]{Besse} or~\cite[pp.~125-126]{Gallot-Hulin-Lafontaine}),
the result follows readily. See also Remark~\ref{rmk:endmanifold} below.

Since
\[
\eta(K_{k})=\dim\,K_{k}=2n-1
\]
(and since clearly the index is constant on the critical set $K_{k}$), we
infer that $K_{k}$ is a nondegenerate Morse-Bott manifold. The manifold
$Y_{k}$ is embedded near $K_{k}$. Since $F=k\pi/2$ on $K_{k}$ and $F\leq
k\pi/2$ on $Y_{k}$ , it follows that the second derivative of $F$ is $\leq0$
on the normal bundle to $K_{k}$ in $Y_{k}$, and at each point $\gamma$ of
$K_{k}$ there is a subspace of $T_{\gamma}\mathcal{P}_{n}$ of dimension $\dim
Y_{k}=(k+1)n$ on which the second derivative of $F$ is $\leq0$. Thus
$\iota(K_{k})+\eta(K_{k})\geq(k+1)n$. But since we know the nullity, it must
be that
\[
\iota(K_{k})\geq\mathrm{codim}(K_{k})=1+(k-1)n.
\]

We will now prove the reverse inequality
\begin{equation}
\iota(K_{k})\leq1+(k-1)n, \label{eq:reverse}
\end{equation}
from which the Lemma follows. Our proof of~\eqref{eq:reverse} goes by
comparing the energy functional on $\mathcal{P}_{n}$ with the energy
functional on the space of free loops in $\mathbb{C}P^{n}$ and making use of
the following well-known facts about the latter. Denote $\mathcal{L}
\mathbb{C}P^{n}:=\{\gamma:S_{\pi}^{1}\rightarrow\mathbb{C}P^{n}\}$ the space
of free loops of Sobolev class $W^{1,2}$ on $\mathbb{C}P^{n}$. The critical
set of the energy functional on $\mathcal{L}\mathbb{C}P^{n}$ for the standard
metric is a disjoint union of Morse-Bott nondegenerate manifolds $L_{m}$,
$m\geq0$ which consist respectively of the geodesics of length $m\pi$ (we
refer to these as being geodesics of \emph{multiplicity $m$}). The index and
nullity of these critical manifolds are respectively equal to $\iota(L_{0}
)=0$, $\eta(L_{0})=2n$ and, for $m\geq1$,~\cite[\S 2.2]{Ziller1977} \
\begin{equation}
\iota(L_{m})=1+(m-1)2n,\qquad\eta(L_{m})=4n-1. \label{eq:indexL}
\end{equation}

\emph{Proof of~\eqref{eq:reverse} for $k$ odd.} Since
inequality~\eqref{eq:reverse} depends only on the behavior of $E$ near $K_{k}
$, let us fix a small enough neighborhood of $K_{k}$ in $\mathcal{P}_{k}$,
denoted $\mathcal{N}_{k}$. There is an embedding
\[
\varphi_{k}:\mathcal{N}_{k}\hookrightarrow\mathcal{L}\mathbb{C}P^{n}
\]
that associates to a path $\gamma$ with endpoints $x_{0}$ and $x_{k}$ the loop
$\gamma\cdot_{\min}C_{x_{k},x_{0}}$, where $\cdot_{\min}$ denotes the minimal
energy concatenation, and $C_{x_{k},x_{0}}$ is the unique vertical half-circle
joining the two distinct points $x_{k}$ and $x_{0}$ such that $\gamma
\cdot_{\min}C_{x_{k},x_{0}}$ is close to a closed geodesic at $x_{0}$ in
$L_{\frac{k+1}{2}}$. Note that $\varphi_{k}(K_{k})=L_{\frac{k+1}{2}}$. If
$V\subset T_{\gamma}\mathcal{P}_{n}$ is a subspace on which the second
derivative of $F$ is negative definite, $d\varphi_{k}V$ is as subspace of
$T_{\varphi_{k}(\gamma)}\mathcal{L}\mathbb{C}P^{n}$ of the same dimension
(because $\varphi_{k}$ is an embedding); since
\begin{align*}
F(\varphi_{k}(\gamma))  &  =F(\gamma)+\frac{\pi}{2}\\
F(\varphi_{k}(\tau))  &  =F(\tau)+F(C_{x_{k},x_{0}})\leq F(\tau)+\frac{\pi}{2}
\end{align*}
for all $\tau\in\mathcal{N}_{k}$, the second derivative of $F$ on
$d\varphi_{k}V$ is also negative definite, and therefore the index of a
geodesic $\gamma\in K_{k}$ in $\mathcal{P}_{n}$ is bounded from above by the
index of the geodesic $\varphi_{k}(\gamma)\in L_{\frac{k+1}{2}}$ in
$\mathcal{L}\mathbb{C}P^{n}$, which is equal to $1+(\frac{k+1}{2}
-1)2n=1+(k-1)n$. This proves~\eqref{eq:reverse}.

\emph{Proof of~\eqref{eq:reverse} for $k$ even.} Given a point $\ast
\in\mathbb{R}P^{n}$ we denote $\mathcal{P}_{\ast,\mathbb{R}P^{n}}
\mathbb{C}P^{n}$ the space of paths $\gamma:[0,1]\rightarrow\mathbb{C}P^{n}$
of Sobolev class $W^{1,2}$ such that $\gamma(0)=\ast$ and $\gamma
(1)\in\mathbb{R}P^{n}$. Let $\Omega\mathbb{C}P^{n}$ be the space of loops of
Sobolev class $W^{1,2}$ in $\mathbb{C}P^{n}$ based at $\ast$. The obvious
inclusions
\[
\Omega\mathbb{C}P^{n}\subset\mathcal{P}_{\ast,\mathbb{R}P^{n}}\mathbb{C}
P^{n}\subset\mathcal{P}_{n}
\]
are codimension $n$ embeddings. Note that any element of $K_{k}$ belongs to
$\mathcal{L}\mathbb{C}P^{n}$, and also to $\mathcal{P}_{\ast,\mathbb{R}P^{n}
}\mathbb{C}P^{n}$ and $\Omega\mathbb{C}P^{n}$ for a suitable choice of
basepoint, and is a critical point of the restriction of the norm (and the
energy function) to each of these spaces. Given $\gamma\in K_{k}$ we denote
$\iota_{\mathcal{P}}(\gamma)$ the index of $\gamma$ as a critical point of the
norm restricted to the submanifold $\mathcal{P}$ of the space of paths on
$\mathbb{C}P^{n}$ . We prove the following relations:
\begin{align}
\iota_{\mathcal{P}_{n}}(\gamma)  &  \leq\iota_{\mathcal{P}_{\ast
,\mathbb{R}P^{n}}\mathbb{C}P^{n}}(\gamma),\label{eq:P-P*}\\
\iota_{\mathcal{P}_{\ast,\mathbb{R}P^{n}}\mathbb{C}P^{n}}(\gamma)  &
\leq\iota_{\Omega\mathbb{C}P^{n}}(\gamma)+n,\label{eq:P*-Omega}\\
\iota_{\Omega\mathbb{C}P^{n}}(\gamma)  &  =\iota_{\mathcal{L}\mathbb{C}P^{n}
}(\gamma). \label{eq:Omega-L}
\end{align}
Since $\gamma\in K_{k}\subset\mathcal{L}\mathbb{C}P^{n}$ has multiplicity
$m=k/2$, we obtain using~\eqref{eq:indexL} that $\iota_{\mathcal{L}
\mathbb{C}P^{n}}(\gamma)=1+(\frac{k}{2}-1)2n=1+(k-2)n$. Thus $\iota
_{\mathcal{P}_{n}}(\gamma)\leq1+(k-2)n+n=1+(k-1)n$, which proves~\eqref{eq:reverse}.

Let us prove~\eqref{eq:P-P*}. The group $\mathrm{Isom}(\mathbb{C}
P^{n},\mathbb{R}P^{n})$ of isometries of $\mathbb{C}P^{n}$ that preserve
$\mathbb{R}P^{n}$ acts transitively on $\mathbb{R}P^{n}$. Let us choose an
open neighborhood $U$ of $\ast$ in $\mathbb{R}P^{n}$ and a smooth map
\[
\psi:U\rightarrow\mathrm{Isom}(\mathbb{C}P^{n},\mathbb{R}P^{n})
\]
such that
\[
\psi(x)(x)=\ast.
\]
Recall the evaluation map $\mathrm{ev}_{0}:\mathcal{P}_{n}\rightarrow
\mathbb{R}P^{n}$ and let $\mathcal{N}:=\mathrm{ev}_{0}^{-1}(U)\subset
\mathcal{P}_{n}$, so that $\mathcal{N}$ is an open neighborhood of
$\mathcal{P}_{\ast,\mathbb{R}P^{n}}\mathbb{C}P^{n}$. We have diffeomorphisms
which are inverse to each other
\[
\xymatrix{
\mathcal N \ \ \ar@<1ex>[rrr]^-P_-{\sim} & & & \ \ \mathcal{P}_{*,\mathbb{R}P^n} \mathbb{C}P^n \times U \ar@<1ex>[lll]^-Q
}
\]
given by
\[
P(\alpha):=(\psi(\mathrm{ev}_{0}(\alpha))\circ\alpha,\mathrm{ev}_{0} 
(\alpha)),\qquad Q(\beta,x):=\psi(x)^{-1}\circ\beta.
\]
Since we use elements $\psi\in\mathrm{Isom}(\mathbb{C}P^{n},\mathbb{R}P^{n})$
the norm is preserved in the sense that, if $\alpha\in\mathcal{N}$, then
\[
F(\alpha)=F\circ pr_{1}\circ P(\alpha)
\]
where $pr_{1}$ is the projection on the first factor of $\mathcal{P} 
_{\ast,\mathbb{R}P^{n}}\mathbb{C}P^{n}\times U$. Now let $V\subset T_{\gamma
}\mathcal{N}$ be a subspace of dimension $\iota$ on which the second
derivative of the norm is negative definite. Then $d(pr_{1}\circ P)(V)$ is a
subspace of T$_{\gamma}\mathcal{P}_{\ast,\mathbb{R}P^{n}}\mathbb{C}P^{n}$ of
dimension $\iota$ on which the second derivative is negative definite. This
implies that $\iota_{\mathcal{N}}(\gamma)\leq\iota_{\mathcal{P}_{\ast
,\mathbb{R}P^{n}}\mathbb{C}P^{n}}(\gamma)$, which is equivalent to~\eqref{eq:P-P*}.

Inequality~\eqref{eq:P*-Omega} is implied by the fact that the embedding
$\Omega\mathbb{C}P^{n}\subset\mathcal{P}_{*,\mathbb{R}P^{n}} \mathbb{C}P^{n}$
has codimension $n$.

The identity~\eqref{eq:Omega-L} is proved by an argument similar to the one
given for~\eqref{eq:P-P*}: using the fact that the group $\mathrm{Isom} 
(\mathbb{C}P^{n})$ of isometries of $\mathbb{C}P^{n}$ acts transitively, we
construct a norm preserving diffeomorphism $\mathcal{N}\simeq\Omega
\mathbb{C}P^{n}\times U$ between an open neighborhood $\mathcal{N} 
:=\mathrm{ev}_{0}^{-1}(U)$ of $\Omega\mathbb{C}P^{n}$ in $\mathcal{L} 
\mathbb{C}P^{n}$ and the product of $\Omega\mathbb{C}P^{n}$ with an open
neighborhood $U$ of $\ast$ in $\mathbb{C}P^{n}$. The relation $F|_{\mathcal{N}
}\equiv F|_{\Omega\mathbb{C}P^{n}}\circ pr_{1}$ implies~\eqref{eq:Omega-L}.
\end{proof}

\begin{remark}
\label{rmk:endmanifold} The statement of Lemma~1 can alternatively be proved
using the Morse index theorem for the endmanifold case
in~\cite{Hingston-Kalish}. The key notions are those of \emph{$\mathbb{R}
P^{n}$-focal point} and \emph{$\mathbb{R}P^{n}$-Jacobi field} along a geodesic
$\gamma$ starting at $x\in\mathbb{R}P^{n}$ in an orthogonal direction. These
generalize the standard notions of conjugate point and Jacobi field. Since the
Jacobi fields in $\mathbb{C}P^{n}$ for the Fubini-Study metric are explicitly
known~\cite[pp.~125-126]{Gallot-Hulin-Lafontaine}, the $\mathbb{R}P^{n}$-focal
points can also be computed explicitly. The key argument in this computation
is the following Lemma, which we leave to the interested reader.

\begin{lemma}
Let $x\in\mathbb{R}P^{n}$ and $\gamma$ be a geodesic in $\mathbb{C}P^{n}$
starting at $x$ with unit speed in a direction orthogonal to $\mathbb{R}P^{n}
$. Denote $X_{1}:=\dot{\gamma}(0)$, $X_{2}:=iX_{1}$, let $(X_{3},X_{5},\dots
X_{2n-1})$ be an orthonormal basis of $\langle X_{1}\rangle^{\perp}$ in
$T_{x}\mathbb{R}P^{n}$ and denote $X_{2j}:=iX_{2j-1}$ for $j\in\{2,\dots,n\}$.
Denote $(X_{1}(s),X_{2}(s),\dots,X_{2n}(s))$ the orthonormal frame at
$\gamma(s)$ obtained from $(X_{1},X_{2},\dots,X_{2n})$ by parallel transport
along $\gamma$. Then
\[
T_{\gamma(\frac{\pi}{2})}\mathbb{R}P^{n}=\langle X_{2}(\frac{\pi}{2}
),X_{3}(\frac{\pi}{2}),X_{5}(\frac{\pi}{2}),\dots,X_{2n-1}(\frac{\pi}
{2})\rangle.
\]

\end{lemma}
\end{remark}

\bigskip

\begin{remark}
Consider the map
\[
\pi_{k}:Y_{k}\rightarrow(\mathbb{R}P^{n})^{k+1},\qquad\gamma\mapsto
(\gamma(0),\gamma(1),\dots,\gamma(k)),
\]
denote $\Delta_{k}:=\{(x_{0},\dots,x_{k})\in(\mathbb{R}P^{n})^{k+1}
\,:\,\exists j,\ x_{j}=x_{j+1}\}\subset(\mathbb{R}P^{n})^{k+1}$, let
$\Delta_{k}^{c}$ be the complement of $\Delta_{k}$ and denote $\dot{Y}_{k}:=\pi_{k}^{-1}(\Delta_{k}^{c})$. Then $\pi_{k}:\dot{Y}_{k}\rightarrow
\Delta_{k}^{c}$ is a $2^{k}:1$ cover and is a $2:1$ cover near $K_{k}
\subset\dot{Y}_{k}$. The second claim follows from the fact that, given a
point $x$ with antipode $x^{\prime}$, the fiber $\pi_{k}^{-1}(x,x^{\prime
},x,\dots)$ contains exactly two geodesics.
\end{remark}


\section{The homology of $\mathcal{P}_{n}$}

\label{sec:homologyPn}

Given $a\geq0$ we denote $\mathcal{P}_{n}^{\leq a}$, $\mathcal{P}_{n}^{<a}$
the sublevel sets $\{F\leq a\}$, respectively $\{F<a\}$.

\begin{theorem}
\label{thm:HPn} The norm functional on $\mathcal{P}_{n}$ for the Fubini-Study
metric on $\mathbb{C}P^{n}$ is a \emph{(strong) perfect Morse function},
meaning that the homology of $\mathcal{P}_{n}$ with arbitrary coefficients is
the direct sum of the \emph{level homology groups} of $F$:
\begin{equation}
H_{\cdot}(\mathcal{P}_{n})\simeq H_{\cdot}(\mathbb{R}P^{n})\oplus
\bigoplus_{k\geq1}H_{\cdot}(\mathcal{P}_{n}^{\leq k\pi/2},\mathcal{P}
_{n}^{<k\pi/2}). \label{eq:HPn}
\end{equation}

\end{theorem}

\begin{remark}
In the literature the term ``perfect Morse function" usually means that the
isomorphism between the homology of the total space and the direct sum of the
level homologies holds with coefficients in any field~\cite{Atiyah-Bott}.
\end{remark}

The proof of the theorem relies on the fact that the $Y_{k}$'s are strong
completing manifolds in the sense of Definition~\ref{def:completing} below.
The latter is reminiscent of~\cite[p.~531]{Atiyah-Bott} and~\cite[p.~97]{Hingston} (see also~\cite[IX.7]{Morse}, \cite[p.~979]{Bott-Samelson}). It
follows of course that the energy functional $E$ is also a (strong) perfect
Morse function.

\begin{definition}
\label{def:completing} Let $X$ be a Hilbert manifold and $f:X\to\mathbb{R}$ a
$C^{2}$-function satisfying condition~(C) of Palais and Smale~\cite[p.~26]{Klingenberg}. 
Let
\[
K:=\mathrm{Crit}(f)\cap f^{-1}(0)\subset X
\]
be the critical locus of $f$ at level $0$ and assume $K$ is a Morse-Bott
nondegenerate manifold of index $\iota(K)$. Denote $X^{\le0}$, $X^{<0}$,
$X^{=0}$ the (sub)level sets of $f$.

A \emph{completing manifold for $K$} is a finite dimensional manifold $Y$
together with a submanifold $L\subset Y$ of codimension $\iota(K)$ and a map
$\varphi:Y\rightarrow X^{\leq0}$ subject to conditions~\eqref{item:embedding}
and~\eqref{item:surjective} below. We say that $Y$ is a \emph{strong
completing manifold} if it satisfies conditions~\eqref{item:embedding}
and~\eqref{item:retraction} (in which case condition~\eqref{item:surjective}
follows). We will call $Y$ a \emph{local completing manifold} if it satisfies
just condition~\eqref{item:embedding}.

\begin{enumerate}
\item \label{item:embedding} the map $\varphi$ is an embedding \emph{near}
$L$, it maps $L$ diffeomorphically onto $K$, and
\[
\varphi^{-1}(K)=L.
\]

\item \label{item:surjective} the canonical map
\[
H_{\cdot}(Y)\to H_{\cdot}(Y,Y\setminus L)
\]
is surjective for any choice of coefficient ring.

\item \label{item:retraction} $Y$ admits a retraction onto $L$.
\end{enumerate}
\end{definition}

\begin{remark}
\label{rmk:completing-perturbed} By an arbitrarily small perturbation of
$\varphi$ along $-\nabla F$ we will have pushed all noncritical points at
level $0$ below level $0$, so that we obtain a map
\[
\widetilde{\varphi}:(Y,Y\setminus L)\rightarrow(X^{\leq0},X^{<0})
\]
satisfying the same conditions as $\varphi$. We can, and shall, assume without
loss of generality that $\widetilde{\varphi}=\varphi$.
\end{remark}

\begin{remark}
\label{rmk:perfect} Let us prove that condition~\eqref{item:retraction}
implies condition~\eqref{item:surjective}. Let $s:L\hookrightarrow Y$ and
$\mathrm{incl}:Y\hookrightarrow(Y,L)$ be the inclusions, denote $\nu$ the
normal bundle to $L$ in $Y$ and let $o_{\nu}$ be the orientation local system
for $\nu$. This is a local system with fiber $\mathbb{Z}$ on $Y$, whose
monodromy along a loop is minus the identity iff the loop is orientation
reversing for $\nu$. Then we have a commutative diagram
\begin{equation}
\label{eq:Thom}\xymatrix{ H_\cdot(Y) \ar[r]^-{\mathrm{incl}_*} \ar[dr]_{s_!} & H_\cdot(Y,Y\setminus L) \ar[d]_-\simeq^-{Thom}\\ & H_{\cdot-\mathrm{codim}(L)}(L;o_\nu) }
\end{equation}
Here the vertical arrow is the Thom isomorphism (composed with excision) and
$s_{!}=PD\circ s^{*}\circ PD$, with $PD$ denoting Poincar\'e duality. In case
we have a retraction $p:Y\to L$, we obtain $ps=\mathrm{Id}_{L}$ and
$s_{!}p_{!}=\mathrm{Id}_{H_{\cdot}(L)}$. Thus $s_{!}$ is surjective, and so is
$\mathrm{incl}_{*}$.
\end{remark}

\begin{lemma}
\label{lem:completing-local} Let $(Y,L,\varphi)$ with $\varphi:Y\rightarrow
X^{\leq0}$ be local completing manifold data for $K=\mathrm{Crit}(f)\cap
f^{-1}(0)$ as in Definition~\ref{def:completing}. Let $\iota$ be the index of
$K$.

(a) the map $\varphi$ induces a canonical morphism
\[
\varphi_{*}:H_{\cdot}(Y,Y\setminus L)\to H_{\cdot}(X^{\le0},X^{<0}).
\]

(b) Given a submanifold $Z\subset Y$ that is transverse to $L$, denote $A:=
Z\cap L$ and let $k$ be the dimension of $A$. The codimension of $A$ in $Z$ is
$\iota$, and the class
\[
\varphi_{*}([Z,Z\setminus A])\in H_{k+\iota}(X^{\leq0},X^{<0})
\]
is the image under the Thom isomorphism $H_{k}(K;o_{\nu^{-}})\overset{\sim
}{\longrightarrow}H_{k+\iota}(X^{\leq0},X^{<0})$ of the class
\[
[A]\in H_{k}(K;o_{\nu^{-}}).
\]

\end{lemma}

\begin{proof}
(a) The morphism $\varphi_{*}$ is well-defined in view of
Remark~\ref{rmk:completing-perturbed}. More precisely, one considers a
perturbation of $\varphi$ of the form $\tilde\varphi_{\varepsilon}
:=\phi^{\varepsilon}_{-\nabla F} \circ\varphi$ where $\phi^{\varepsilon
}_{-\nabla F}$, $\varepsilon>0$ is the time-$\varepsilon$ flow of $-\nabla F$.
All the maps $\tilde\varphi_{\epsilon}$ are homotopic and act as
$(Y,Y\setminus L)\to(X^{\le0},X^{<0})$.

(b) The key point is identifying the normal bundle. Transversality implies
that the normal bundle to $A$ in $Z$ is isomorphic to the restriction of
$\nu^{-}$ to $A$, and the conclusion follows.
\end{proof}

\begin{lemma}
\label{lem:completing} Let $X$, $f$, and $K$ be as in
Definition~\ref{def:completing} and denote $\nu^{-}$ the negative bundle of
$K$ (of rank $\iota(K)$). If $K$ admits a completing manifold then we have
short exact sequences
\[
0\rightarrow H_{\cdot}(X^{<0})\rightarrow H_{\cdot}(X^{\leq0})\rightarrow
H_{\cdot-\iota(K)}(K;o_{\nu^{-}})\rightarrow0.
\]
If $K$ admits a strong completing manifold these exact sequences are split, so
that
\[
H_{\cdot}(X^{\leq0})\simeq H_{\cdot}(X^{<0})\oplus H_{\cdot-\iota(K)}
(K;o_{\nu^{-}}).
\]

\end{lemma}

\begin{proof}
Let $(Y,L,\varphi)$ be the completing manifold data for $K$. In view of
Remark~\ref{rmk:completing-perturbed} we can assume without loss of generality
that $\varphi$ acts as
\[
\varphi:(Y,Y\setminus L)\to(X^{\le0},X^{<0}).
\]
By functoriality of the long exact sequence of a pair we obtain a commutative
diagram
\[
\xymatrix
@C=40pt {H_{\cdot}(X^{<0})\ar[r]  &  H_{\cdot}(X^{\leq0})\ar[r]^-
{(\mathrm{incl}_{X})_{\ast}}  &  H_{\cdot}(X^{\leq0},X^{<0}) \\
H_{\cdot}(Y\setminus L)\ar[r]\ar[u]  &  H_{\cdot}(Y)\ar@{->>}[r]^-
{\mathrm{incl}_{\ast}}\ar[u]_{\varphi_{\ast}}  &  H_{\cdot}(Y,Y\setminus
L)\ar[u]_{\varphi_{\ast}^{rel}}^{\simeq}}
\]
The map $\mathrm{incl}_{\ast}$ is surjective by assumption. That the map
$\varphi_{\ast}^{rel}$ is an isomorphism follows from the fact that $\varphi$
is an embedding near $L$, maps $L$ diffeomorphically onto the nondegenerate
critical manifold $K$, and maps $Y\setminus L$ into $X^{<0}$, so that
\[
(\varphi|_{L})^{\ast}\nu^{-}\simeq\nu.
\]
Thus $\varphi_{\ast}^{rel}$ can be written as the composition
\[
H_{\cdot}(Y,Y\setminus L)\simeq H_{\cdot-\iota(K)}(L;o_{\nu})\simeq
H_{\cdot-\iota(K)}(K;o_{\nu^{-}})\simeq H_{\cdot}(X^{\leq0},X^{<0}),
\]
where the first map is the Thom isomorphism composed with (the inverse of)
excision, and the third map follows from Morse-Bott
theory~\cite{Bott-nondegenerate}.

Thus $(\mathrm{incl}_{X})_{\ast}$ is surjective, which implies the exactness
of the short sequence in the statement. In case $Y$ is a strong completing
manifold, we obtain a section of $(\mathrm{incl}_{X})_{\ast}$ as the
composition $\varphi_{\ast}\circ\sigma\circ(\varphi_{\ast}^{rel})^{-1}$, where
$\sigma:H_{\cdot}(Y,Y\setminus L)\rightarrow H_{\cdot}(Y)$ is the section of
$\mathrm{incl}_{\ast}$ constructed in Remark~\ref{rmk:perfect}.
\end{proof}

The distinction between strong completing manifolds and completing manifolds
is null if one restricts to field coefficients. In this situation all short
exact sequences split and thus the direct sum decomposition in
Lemma~\ref{lem:completing} is automatic.

\begin{proof}
[Proof of Theorem~\ref{thm:HPn}]Recall the map $\varphi_{k}:Y_{k}
\to\mathcal{P}_{n}$ and the submanifold $L_{k}\subset Y_{k}$
from~\eqref{eq:varphik} and~\eqref{eq:KkYk}. We now show that the triple
$(Y_{k},L_{k},\varphi_{k})$ is a strong completing manifold for $K_{k}$, from
which the Theorem follows in view of Lemma~\ref{lem:completing}.

The map $\varphi_{k}$ correctly sends $Y_{k}$ into $\mathcal{P}_{n}^{\leq
k\pi/2}$. To check Condition~\eqref{item:embedding} in
Definition~\ref{def:completing}, let us first note that the map $\varphi_{k}$
is an embedding outside the locus of points in $Y_{k}$ for which at least one
of the circles $C_{x_{j},v_{j} ,\theta_{j}}$ is constant, and in particular
$\varphi_{k}$ is an embedding near $L_{k}$. Secondly, the map $\varphi_{k}$
sends by definition $L_{k}$ diffeomorphically onto $K_{k}$. Finally, it
follows from the definition of $\varphi_{k}$ that $\varphi_{k}^{-1}
(\mathcal{P}_{n}^{=k\pi/2})$ consists of $k$-tuples of great half-circles of
length $\pi/2$ with matching endpoints, and therefore the only critical points
at level $k\pi/2$ lie in $L_{k}$.

To check Condition~\eqref{item:retraction} in Definition~\ref{def:completing}
we consider the map
\[
p_{k}:Y_{k}\to SN\mathbb{R}P^{n}, \qquad\big( (x_{0},v_{0},\theta_{0}
),\dots,(x_{k-1},v_{k-1},\theta_{k-1})\big)\longmapsto(x_{0},v_{0}).
\]
Via the identifications $SN\mathbb{R}P^{n}\equiv K_{k}\equiv L_{k}$ we can
view this map as $p_{k}:Y_{k}\to L_{k}$. Its composition with the inclusion
$L_{k}\hookrightarrow Y_{k}$ is clearly the identity map of $L_{k}$, hence
$p_{k}$ is a retraction of $Y_{k}$ onto $L_{k}$.
\end{proof}

\begin{remark}
\label{rmk:pk} The map $p_{k}:Y_{k}\to SN\mathbb{R}P^{n}$ actually defines a
fiber bundle structure on $Y_{k}$. To see this we note that $p_{k}$ can be
written as a composition
\[
Y_{k}\to Y_{k-1}\to\ldots\to Y_{1}\to SN\mathbb{R}P^{n}.
\]
Here $Y_{j}\to Y_{j-1}$, $2\le j\le k$ is the projection on the first
$(j-1)$-components of any $j$-tuple. This defines a fiber bundle structure on
$Y_{j}$ with base $Y_{j-1}$ and fiber $S^{n-1}\times S^{1}_{\pi}$, where
$S^{n-1}$ is the sphere of radius $1$ in $\mathbb{R}^{n}$. The map $Y_{1}\to
SN\mathbb{R}P^{n}$ is the projection on the first component, recalling that
$Y_{1}= SN\mathbb{R}P^{n}\times S^{1}_{\pi}$. As a consequence, $p_{k}$ is a
fiber bundle with fiber $(S^{n-1})^{k-1}\times(S^{1}_{\pi})^{k}$.
\end{remark}

The following statement is a rephrasing of Theorem~\ref{thm:HPn}.

\begin{corollary}
\label{cor:HPn} Denote $o_{\nu^{-}_{k}}$, $k\ge1$ the local system of
orientations for the negative bundle to $K_{k}$ and view this as a local
system on $ST\mathbb{R}P^{n}$ via the identifications $ST\mathbb{R}P^{n}\equiv
SN\mathbb{R}P^{n}\equiv K_{k}$. We have an isomorphism of graded 
$\mathbb{Z}$-modules
\begin{equation}
\label{eq:HPn-cor}H_{\cdot}(\mathcal{P}_{n})\simeq H_{\cdot}(\mathbb{R}
P^{n})\oplus\bigoplus_{k\ge1} 
H_{\cdot}(ST\mathbb{R}P^{n};o_{\nu^{-}_{k}})[-1-(k-1)n].
\end{equation}

\end{corollary}

\vspace{-.4cm}\hfill{$\square$}

We will now compute these graded $\mathbb{Z}$-modules explicitly.

\begin{lemma}
\label{lem:localsystems} \quad

(a.1) If $n$ is odd, the local system $o_{\nu_{k}^{-}}$ is trivial for all
$k\geq1$.

(a.2) If $n$ is even, the local system $o_{\nu_{k}^{-}}$ is trivial for odd
$k\geq1$, and it is nontrivial for even $k\geq2$. In the second case we have
\begin{equation}
\label{eq:X}o_{\nu_{k}^{-}}=\pi^{\ast}o,
\end{equation}
where $o$ is the orientation local system on $\mathbb{R}P^{n}$.

(b.1) If $n$ is odd, then
\begin{equation}
\label{eq:Y}H_{\cdot}(ST\mathbb{R}P^{n};\mathbb{Z})=H_{\cdot}(\mathbb{R}
P^{n};\mathbb{Z})\otimes H_{\cdot}(S^{n-1};\mathbb{Z}).
\end{equation}

(b.2) If $n$ is even, then
\begin{equation}
\label{eq:Z}H_{\cdot}(ST\mathbb{R}P^{n};\mathbb{Z})=\left\{
\begin{array}
[c]{ll}
H_{\cdot}(\mathbb{R}P^{n};\mathbb{Z}), & \mbox{in degrees }\leq n-2,\\
\mathbb{Z}/4, & \mbox{in degree }n-1,\\
H_{\cdot}(\mathbb{R}P^{n};o)[-n+1], & \mbox{in degrees lying in }\{n,\dots
,2n-1\},
\end{array}
\right.
\end{equation}
and
\begin{equation}
\label{eq:W}H_{\cdot}(ST\mathbb{R}P^{n};\pi^{\ast}o)=H_{\cdot}(\mathbb{R}
P^{n};o)\oplus H_{\cdot}(\mathbb{R}P^{n};\mathbb{Z})[-n+1].
\end{equation}

\end{lemma}

\begin{proof}
The homology of $ST\mathbb{R}P^{n}$ with coefficients in a local system which
is pulled back from $\mathbb{R}P^{n}$ can be computed using the Leray-Serre
spectral sequence for the bundle $S^{n-1}\hookrightarrow ST\mathbb{R}
P^{n}\overset{\pi}{\longrightarrow}\mathbb{R}P^{n}$. Before discussing
nontrivial local systems on $ST\mathbb{R}P^{n}$, we will compute $H_{\cdot
}(ST\mathbb{R}P^{n};\mathbb{Z})$.

Let $n$ be odd and consider the spectral sequence for $H_{\cdot}
(ST\mathbb{R}P^{n};\mathbb{Z})$. The second page is $H_{\cdot}(\mathbb{R}
P^{n};\mathbb{Z})\otimes H_{\cdot}(S^{n-1};\mathbb{Z})$ and $d^{n}=0$ since
the Euler class of $\mathbb{R}P^{n}$ is zero. The differentials $d^{r}$,
$r\neq n$ vanish for dimension reasons. Recalling that the integral homology
of $\mathbb{R}P^{n}$ in (ascending) degree $\leq n$ is $H_{\cdot}(\mathbb{R}
P^{n};\mathbb{Z})=(\mathbb{Z},\mathbb{Z}/2,0,\dots,\mathbb{Z}/2,0,\mathbb{Z})$
and using that $Ext(\mathbb{Z},\mathbb{Z}/2)=0$ we obtain~\eqref{eq:Y}.

Now let $n$ be even and consider the spectral sequence for $H_{\cdot
}(ST\mathbb{R}P^{n};\mathbb{Z})$. The second page is $H_{\cdot}(\mathbb{R}
P^{n};\mathbb{Z})\oplus H_{\cdot}(\mathbb{R}P^{n};o)[-n+1]$. The integral
homology of $\mathbb{R}P^{n}$ in (ascending) degree $\leq n$ is $H_{\cdot
}(\mathbb{R}P^{n};\mathbb{Z})=(\mathbb{Z},\mathbb{Z}/2,0,\dots,\mathbb{Z}
/2,0)$, whereas the homology with coefficients in $o$ in (ascending) degree
$\leq n$ is $(\mathbb{Z}/2,0,\dots,\mathbb{Z}/2,0,\mathbb{Z})$. Here we use
the Poincar\'{e} duality isomorphism $H_{\cdot}(\mathbb{R}P^{n};o)\simeq
H^{n-\cdot}(\mathbb{R}P^{n};\mathbb{Z})$. The differential $d^{n}$ is thus $0$
since all its components have either vanishing source or vanishing target,
whereas the differentials $d^{r}$, $r\neq n$ vanish for dimension reasons.
This determines unambiguously the homology $H_{\cdot}(ST\mathbb{R}P^{n}
;\mathbb{Z})$ in all degrees except $n-1$, where we know that it is a
$\mathbb{Z}/2$-extension of $\mathbb{Z}/2$. By writing the spectral sequence
with $\mathbb{Z}/2$-coefficients and using that the Euler class of
$\mathbb{R}P^{n}$ is equal to $1$, we obtain that $H_{n-1}(ST\mathbb{R}
P^{n};\mathbb{Z}/2)=\mathbb{Z}/2$, and thus $H_{n-1}(ST\mathbb{R}
P^{n};\mathbb{Z})=\mathbb{Z}/4$. Thus we have~\eqref{eq:Z}.

In particular, we obtain that the fundamental group of $ST\mathbb{R}P^{n}$ is
\[
\pi_{1}(ST\mathbb{R}P^{n})=\left\{
\begin{array}
[c]{c}
\mathbb{Z}/4\text{ \ if \ }n=2\\
\mathbb{Z}/2\text{ \ if \ }n\geq3
\end{array}
\right\}  .
\]
Indeed, the fundamental group is abelian since $ST\mathbb{R}P^{n}$ is a
$\mathbb{Z}/2$-quotient of $STS^{n}$, which has cyclic fundamental group, and
is therefore isomorphic to $H_{1}(ST\mathbb{R}P^{n};\mathbb{Z})$.

We now move on to nontrivial systems on $ST\mathbb{R}P^{n}$. Note that
$\mathbb{R}P^{n}$, $n\geq1$ carries a unique nontrivial local system, denoted
$\tau$, because its fundamental group has a single generator (of even order).
If $n$ is even, $\mathbb{R}P^{n}$ is not orientable and thus $\tau=o$. When
$n$ is odd, $o$ is trivial. While on $ST\mathbb{R}P^{1}=S^{1}\sqcup S^{1}$
there are two nontrivial local systems (modulo diffeomorphisms), on
$ST\mathbb{R}P^{n}$, $n\geq2$ there is a unique nontrivial local system since
the fundamental group still has a single generator which has even order. The
map $\pi_{\ast}$ sends a generator of the fundamental group of the total space
of the bundle to a generator of the fundamental group of the base. Thus the
unique notrivial local system on $ST\mathbb{R}P^{n}$, $n\geq2$ is equal to
$\pi^{\ast}\tau$.

Now let $n$ be even and consider the spectral sequence for $H_{\cdot
}(ST\mathbb{R}P^{n};\pi^{\ast}o)$. The second page is $H_{\cdot}
(\mathbb{R}P^{n};o)\oplus H_{\cdot}(\mathbb{R}P^{n};o\otimes o)[-n+1]$. Note
that $o\otimes o$ is the trivial local system $\mathbb{Z}$. The differential
$d^{n}$ necessarily vanishes due to the specific values of the homology
groups, the differentials $d^{r}$, $r\neq n$ vanish for dimension reasons, and
there are no extension issues since $Ext(\mathbb{Z},\mathbb{Z}/2)=0$. Thus we have~\eqref{eq:W}.

We now prove statements \emph{(a.1)} and~\emph{(a.2)}. Note that~\eqref{eq:X}
is a direct consequence of the fact that the unique nontrivial local system on
$ST\mathbb{R}P^{n}$, $n\ge2$ is $\pi^{*}\tau$, and $\tau=o$ is $n$ is even.
Recall for the proof that we have denoted $L_{k}\subset Y_{k}$ the
diffeomorphic image of $K_{k}$ under the embedding~\eqref{eq:KkYk}. We view
$o_{\nu_{k}^{-}}$ as the local system of orientations for the normal bundle
$\nu_{k}$ to $L_{k}$ in $Y_{k}$ and denote it as such by $o_{\nu_{k}}$. We
identify $L_{k}$ with $ST\mathbb{R}P^{n}$ as above.

The manifold $\mathbb{R}P^{n}$ is orientable if and only if $n$ is odd, but
$ST\mathbb{R}P^{n}$ is always orientable. Let $\widetilde{Y}_{k}$ be an open
neighborhood of $L_{k}$ in $Y_{k}$ that retracts onto $L_{k}$ and is small
enough to be a manifold. Then $o_{\nu_{k}}$ is trivial if and only if
$\widetilde{Y}_{k}$ is orientable, and it is sufficient to check the value of
the orientation system of $\widetilde{Y}_{k}$ on the image of one of the
generators of $\pi_{1}(ST\mathbb{R}P^{1})$. We will use the generator
represented by the tangent vector to the loop (in homogeneous coordinates)
$\gamma(t)=[\cos t:\sin t:0:...:0]$, $0\leq t\leq\pi$ in $\mathbb{R}P^{n}$.
Recall that $\widetilde{Y}_{k}$ is parametrized by sequences 
$(x_{0},...,x_{k})\in(\mathbb{R}P^{n})^{k+1}$ where the $x_{i}$ are the endpoints of
the $k$ half-circles making up a path in the image of $\widetilde{Y}_{k}$. For
each $t\in[0,\pi]$ the element of $L_{k}$ corresponding to the tangent vector
$\dot\gamma(t)\in ST\mathbb{R}P^{n}$ is parametrized by the sequence
\[
(x_{0},...,x_{k})(t)=(\gamma(t),\dot\gamma(t),\gamma(t),\dot\gamma
(t),...)\in(\mathbb{R}P^{n})^{k+1}.
\]
Note that in each case the path $x_{i}(t)$ represents a generator of $\pi
_{1}\mathbb{R}P^{n}$. The value of the orientation system $o^{\otimes k+1}$ on
$(\mathbb{R}P^{n})^{k+1}$ on this loop in $ST\mathbb{R}P^{n}$ is $-1$ if and
only if $n$ and $k$ are even. Thus we have~\emph{(a.1)} and~\emph{(a.2)}.
\end{proof}

\begin{remark}
Here is an alternative proof of statements~\emph{(a.1)} and~\emph{(a.2)}. We
use the global involution
\[
A:Y_{k+1}\to Y_{k+1}
\]
described heuristically as reversing the time direction on the path
represented by an element of $Y_{k+1}$. The involution $A$ is a diffeomorphism
which acts by sending a $(k+1)$-tuple $\big((x_{j},v_{j},\theta_{j}
)\big)_{0\le j\le k}$ to $\big((x^{\prime}_{j},v^{\prime}_{j},\theta^{\prime
}_{j})\big)_{0\le j\le k}$, with $(x^{\prime}_{j},v^{\prime}_{j})$ the
\emph{opposite} of the endpoint speed vector of $C_{x_{k-j},v_{k-j}
,\theta_{k-j}}$ and $\theta_{j}:=-\theta_{k-j}$. Our convention is that the
endpoint speed vector of a constant circle $C_{x_{k-j},v_{k-j},0}$ is
$(x_{k-j},-v_{k-j})$, so that the endpoint speed vector defines a smooth
function on $Y_{1}$. In particular $(x^{\prime}_{j},v^{\prime}_{j}
)=(x_{k-j},v_{k-j})$. Denoting by $\mathrm{ev}_{0}$, $\mathrm{ev}_{1}
:Y_{k+1}\to\mathbb{R}P^{n}$ the evaluation maps at the endpoints, we have
\[
\mathrm{ev}_{1}=\mathrm{ev}_{0}\circ A.
\]

To prove~\emph{(a.1)} and~\emph{(a.2)} we think of $L_{k}$ as being identified
with $SN\mathbb{R}P^{n}$ as above. We also recall the notation $\nu_{k}$ for
the normal bundle to $L_{k}$ in $Y_{k}$. Since $Y_{1}=SN\mathbb{R}P^{n}\times
S^{1}_{\pi}$ we obtain that $\nu_{1}$ is trivial, and so is $o_{\nu_{1}}$.

Let now $k\ge2$. Write $Y_{k}= Y_{k-1} \, {_{\mathrm{ev}_{1}}}\!\!\times
_{\mathrm{ev}_{0}} Y_{1}$ and embed $Y_{k-1}\hookrightarrow Y_{k}$ by adding
to any $(k-1)$-tuple $\big((x_{j},v_{j},\theta_{j})\big)_{0\le j\le k-2}$ the
geodesic half-circle $(x_{k-1},v_{k-1},\pi/2)$, where $(x_{k-1},v_{k-1})$ is
the endpoint speed vector of $C_{x_{k-2},v_{k-2},\theta_{k-2}}$. This
embedding is a section for the fiber bundle $Y_{k}\to Y_{k-1}$. The normal
bundle to $Y_{k-1}$ in $Y_{k}$ under this embedding then satisfies
\[
\nu_{_{Y_{k}}} Y_{k-1} \oplus\mathbb{R} = (\mathrm{ev}_{1})^{*} (N\mathbb{R}
P^{n}\oplus\mathbb{R}).
\]
In this equality we think of the trivial rank one factor on the left hand side
as being generated by the section $v_{k-1}$ along $Y_{k-1}$, and use the
identity $N_{x_{k-1}} \mathbb{R}P^{n}=\langle v_{k-1}\rangle\oplus 
T_{v_{k-1}}SN_{x_{k-1}}\mathbb{R}P^{n}$. The trivial rank one factor on the right hand
side corresponds to varying the argument $\theta_{k-1}\in S^{1}_{\pi}$. We
obtain at the level of orientation local systems
\[
o_{\nu_{_{Y_{k}}} Y_{k-1}} = (\mathrm{ev}_{1})^{*} o_{N\mathbb{R}P^{n}} =
(\mathrm{ev}_{1})^{*} o_{T\mathbb{R}P^{n}}.
\]
Since $\mathrm{ev}_{1}=\mathrm{ev}_{0}\circ A$ and since we only consider
local systems up to diffeomorphisms, we can alternatively write
\[
o_{\nu_{_{Y_{k}}} Y_{k-1}} = (\mathrm{ev}_{0})^{*} o_{T\mathbb{R}P^{n}}.
\]
By induction we then obtain
\[
o_{\nu_{k}} = (\mathrm{ev}_{0})^{*} (o_{T\mathbb{R}P^{n}})^{\otimes k-1}.
\]
Note that $\mathrm{ev}_{0}=\pi$ on $Y_{1}$. Assertions~\emph{(a.1)}
and~\emph{(a.2)} then follow from the fact that $o_{T\mathbb{R}P^{n}}$ is
trivial if $n$ is odd (orientable case), respectively nontrivial if $n$ is
even (nonorientable case).
\end{remark}


\section{The Pontryagin-Chas-Sullivan product}

\label{sec:product}

In this section we restrict to homology with $\mathbb{Z}/2$-coefficients. Our
motivation is that the Pontryagin-Chas-Sullivan product extends the
intersection form on the homology of $\mathbb{R}P^{n}$, and the latter is most
natural with $\mathbb{Z}/2$-coefficients. Our purpose is to prove
Theorem~\ref{thm:PCSproduct}.

Formula~\eqref{eq:HPn-cor} simplifies to
\begin{equation}
H_{\cdot}(\mathcal{P}_{n};\mathbb{Z}/2)\simeq H_{\cdot}(\mathbb{R}
P^{n};\mathbb{Z}/2)\oplus\bigoplus_{k\ge1} H_{\cdot}(ST\mathbb{R}
P^{n};\mathbb{Z}/2)[-1-(k-1)n] \label{eq:HPn-Z2}
\end{equation}
since all local systems become trivial over $\mathbb{Z}/2$. The second page of
the spectral sequence for computing $H_{\cdot}(ST\mathbb{R}P^{n}
;\mathbb{Z}/2)$ simplifies to $H_{\cdot}(\mathbb{R}P^{n};\mathbb{Z}/2)\otimes
H_{\cdot}(S^{n-1};\mathbb{Z}/2)$, i.e.
\[
\xymatrix
@R=5pt { \mbox{\tiny$n-1$}  &  \mathbb{Z}/2  &  \mathbb{Z}/2  &  \ldots &
\mathbb{Z}/2  &  \mathbb{Z}/2 \\ \qquad\\ \qquad\\
\mbox{\tiny$0$}  &  \mathbb{Z}/2  &  \mathbb{Z}/2  &  \ldots &  \mathbb{Z}/2
&  \mathbb{Z}/2 \ar[uuullll]_{\cdot e} \\  &  \mbox{\tiny$0$}  &
\mbox{\tiny$1$}  &  \ldots &  \mbox{\tiny$n-1$}  &  \mbox{\tiny$n$} }
\]
The only depicted differential is multiplication by the Euler number modulo
$2$, which vanishes for $n$ odd and is an isomorphism for $n$ even. Thus for
$n$ odd
\[
H_{\cdot}(ST\mathbb{R}P^{n};\mathbb{Z}/2)=H_{\cdot}(\mathbb{R}P^{n}
;\mathbb{Z}/2)\otimes H_{\cdot}(S^{n-1};\mathbb{Z}/2),
\]
whereas for $n$ even
\[
H_{\cdot}(ST\mathbb{R}P^{n};\mathbb{Z}/2)=\left\{
\begin{array}
[c]{ll}
\mathbb{Z}/2, & \mbox{in degrees lying in }\{0,\dots, 2n-1\},\\
0, & \mbox{else.}
\end{array}
\right.
\]

\noindent\textbf{Convention.} In the rest of this section we shall not include
anymore the coefficient ring $\mathbb{Z}/2$ in the notation for homology
groups. As an example, we shall write $H_{\cdot}(ST\mathbb{R}P^{n})$ instead
of $H_{\cdot}(ST\mathbb{R}P^{n};\mathbb{Z}/2)$.

It is convenient to write the homology $H_{\cdot}(\mathcal{P}_{n})$ as given
by Theorem~\ref{thm:HPn} and formula~\eqref{eq:HPn-Z2} in a table in which the
homological degree appears as the vertical coordinate, and the critical values
of the norm $F$ appear as the horizontal coordinate. \vspace{2cm}

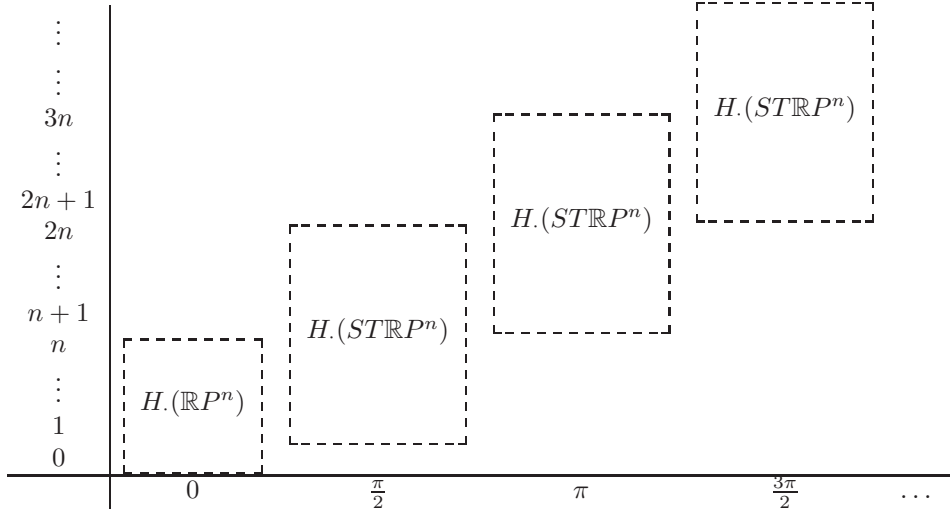
\begin{center}
\begin{tabular}
[c]{c|ccccc}
$\vdots$ &  &  &  &  & \\
$\vdots$ &  &  &  &  & \\
$3n$ &  &  &  &  & \\
$\vdots$ &  &  &  &  & \\
$2n+1$ &  &  &  &  & \\
$2n$ &  &  &  &  & \\
$\vdots$ &  &  &  &  & \\
$n+1$ &  &  &  &  & \\
$n$ &  &  &  &  & \\
$\vdots$ &  &  &  &  & \\
$1$ &  &  &  &  & \\
$0$ &
\raisebox{0pt}[0pt][0pt]{\raisebox{40pt}{\dbox{$\xymatrix@R=8pt{ \\ H_\cdot(\mathbb{R}P^n) \\ \quad }$}}} &
\raisebox{0pt}[0pt][0pt]{\raisebox{83pt}{\dbox{$\xymatrix@R=24pt{ \\ H_\cdot(ST\mathbb{R}P^n) \\ \quad }$}}} &
\raisebox{0pt}[0pt][0pt]{\raisebox{125pt}{\dbox{$\xymatrix@R=24pt{ \\ H_\cdot(ST\mathbb{R}P^n) \\ \quad }$}}} &
\raisebox{0pt}[0pt][0pt]{\raisebox{167pt}{\dbox{$\xymatrix@R=24pt{ \\ H_\cdot(ST\mathbb{R}P^n) \\ \quad }$}}} &
\\\hline
& $0$ & $\frac\pi2$ & $\pi$ & $\frac{3\pi} 2$ & \dots
\end{tabular}
\end{center}

\subsection{Generalities.}

\label{sec:gen}

Our computation of the multiplicative structure of the ring $(H_{\cdot
}(\mathcal{P}_{n}),\ast)$ makes use of several general principles which are of
larger interest and which we now emphasize. Whereas the discussion in (1--3)
below is valid for arbitrary path spaces, item (4) is specific to the pair
$(\mathbb{C}P^{n},\mathbb{R}P^{n})$.

\smallskip\emph{(1) Geometric realization of product classes.}

For the next Lemma we recall the codimension $n$ submanifold 
$\mathcal{C}_{n}\subset\mathcal{P}_{n}\times\mathcal{P}_{n}$ and the concatenation map
$c:\mathcal{C}_{n}\to\mathcal{P}_{n}$.

\begin{lemma}
\label{lem:geom*} Let $[f],[g]\in H_{\cdot}(\mathcal{P}_{n})$ be two classes
which are respectively represented by $C^{1}$-maps $f:A\to\mathcal{P}_{n}$,
$g:B\to\mathcal{P}_{n}$ whose sources are closed finite-dimensional manifolds
$A$, $B$. Assume that the maps
\[
\mathrm{ev}_{1}\circ f:A\to\mathbb{R}P^{n},\qquad\mathrm{ev}_{0}\circ
g:B\to\mathbb{R}P^{n}
\]
are transverse. Denote their fiber product by
\[
C:= A \ _{\mathrm{ev}_{1}\circ f} {\times}_{\mathrm{ev}_{0}\circ g} \ B,
\]
with the natural map $(f,g):C\to\mathcal{C}_{n}$. Then
\[
[f]\ast[g]=[c\circ(f,g)].
\]

\end{lemma}

\begin{proof}
Recall the Alexander-Whitney map $AW:H_{\cdot}(\mathcal{P}_{n})\otimes
H_{\cdot}(\mathcal{P}_{n})\to 
H_{\cdot}(\mathcal{P}_{n}\times\mathcal{P}_{n})$, the inclusion $s:\mathcal{C}_{n}\to\mathcal{P}_{n}\times\mathcal{P}_{n}$
and the fact that $[f]\ast[g]=c_{*}s_{!}A([f]\otimes[g])$. The transversality
assumption in the statement is equivalent to the fact that the map
$(f,g):A\times B\to\mathcal{P}_{n}\times\mathcal{P}_{n}$ is transverse to
$\mathcal{C}_{n}$. Thus $s_{!}([f,g)]=[(f,g):C\to\mathcal{C}_{n}]$. In view of
the equality $AW([f]\otimes[g])=[(f,g)]$, the Lemma follows.
\end{proof}

Here is a direct consequence of Lemma~\ref{lem:geom*}. Denote $[Y_{k}]\in
H_{(k+1)n}(\mathcal{P}_{n})$, $k\ge1$ the class represented by the
``evaluation map'' $\varphi_{k}:Y_{k}\to\mathcal{P}_{n}$. Then
\begin{equation}
\label{eq:YkY1k}[Y_{k}]=[Y_{1}]^{\ast k}.
\end{equation}
Indeed, one proceeds by induction on $k\ge1$ using Lemma~\ref{lem:geom*},
which can be applied since $\mathrm{ev}_{1}\circ\varphi_{k-1}:Y_{k-1}
\to\mathbb{R}P^{n}$ is a submersion. (The same holds, of course, for
$\mathrm{ev}_{0}\circ\varphi_{k-1}:Y_{k-1}\to\mathbb{R}P^{n}$.)

\bigbreak

\smallskip\emph{(2) Min-max critical levels.}

Let $X$ be a Hilbert manifold and $f:X\rightarrow\mathbb{R}$ a function of
class $C^{2}$ satisfying condition~(C) of Palais and Smale. Given a homology
class $\alpha\in H_{\cdot}(X)$, the ``minimax" \emph{ critical level of
$\alpha$ with respect to $f$} is
\[
\mathrm{Crit}(\alpha;f):=\min\{c\in\mathbb{R}\ :\ \alpha\in\mathrm{im}
\big(H_{\cdot}(\{f\leq c\})\rightarrow H_{\cdot}(X)\big)\}.
\]
If there is no danger of confusion, we shall write $\mathrm{Crit}(\alpha)$
instead of $\mathrm{Crit}(\alpha;f)$.

We clearly have
\begin{equation}
\label{eq:critalphabeta}\mathrm{Crit}(\beta)\leq\max(\mathrm{Crit}
(\alpha),\mathrm{Crit}(\alpha+\beta)).
\end{equation}

\begin{corollary}
\label{cor:crit1} \textit{Suppose }$H_{\cdot}(X)$ \textit{is spanned over
}$\mathbb{Z}/2$\textit{ by subspaces }$A$\textit{ and }$B$\textit{, and let
}$\delta\in H_{\cdot}(X)$\textit{. If there is a compact set }$Y\subset
X$\textit{ supporting }$A$\textit{ and }$\delta$\textit{, but no nonzero
element of }$B$\textit{, then }$\delta\in A$.
\end{corollary}

\begin{proof}
There is a smooth, nonnegative function $\varphi:X\rightarrow\mathbb{R} $
whose $0$-set is $Y$. Our hypotheses imply that $\mathrm{Crit} (\alpha
;\varphi)=0$ for all $\alpha\in A$, and $\mathrm{Crit}(\delta;\varphi)=0$. If
$\delta=\alpha+\beta$, then equation~\eqref{eq:critalphabeta} implies
$\mathrm{Crit}(\beta;\varphi)=0$, and thus $\beta=0$.
\end{proof}

Now let $f:X\rightarrow\mathbb{R}$ be a function as above, assume in addition
that it is Morse-Bott, and let $\varphi:\mathrm{Crit}(f)\rightarrow\mathbb{R}$
be a Morse-Bott function of class $C^{2}$ satisfying the Palais-Smale
condition~(C). Let $K$ be a connected component of the critical set of $f$ at
level $c$, with index $\iota$. We allow the critical set of $f$ at level $c$
to be disconnected and the index to vary from one component to the other. For
simplicity we use $\mathbb{Z} /2$-coefficients so that the local system of
orientations for the negative bundle on $K$ is trivial, but this discussion
generalizes in a straightforward way to arbitrary coefficients. Let
$\tau:H_{\cdot-\iota}(K)\rightarrow H_{\cdot}(X^{\leq c},X^{<c})$ be induced
from the Thom isomorphism. As an immediate consequence of
Corollary~\ref{cor:crit1} we have:

\begin{corollary}
\label{cor:crit2} Assume that $H_{\cdot}(K)$ is spanned over $\mathbb{Z}/2$ by
subspaces $A$ and $B$, and that $\delta\in H_{\cdot}(K)$. If there is a
compact set $Y\subset K$ supporting $A$ and $\delta$, but no nonzero element
of $B$, then $\tau(\delta)\in\tau(A)$.

In particular, if every nonzero element of $B$ has nonzero intersection
product with an element of $H_{\cdot}(K)$ that can be supported in $K-Y$, then
$\tau(\delta)\in\tau(A).$
\end{corollary}

We shall not use in the sequel the second part of Corollary~\ref{cor:crit2}.

Critical levels are well-behaved with respect to the Pontryagin-Chas-Sullivan
product provided one uses the \emph{norm}
\[
F:=\sqrt{E}.
\]

More precisely, we have the following

\begin{lemma}
\label{lem:Crit} Let $\alpha,\beta\in H_{\cdot}(\mathcal{P}_{n})$. Min-max
critical levels of the norm functional $F$ satisfy the inequality
\begin{equation}
\mathrm{Crit}(\alpha\ast\beta;F)\leq\mathrm{Crit}(\alpha;F)+\mathrm{Crit}
(\beta;F). \label{eq:Crit}
\end{equation}

\end{lemma}

\hfill{$\square$}

The proof of~\eqref{eq:Crit} given in~\cite[Proposition~5.3]{Goresky-Hingston}
for the case of the free loop space holds verbatim in our situation. Note that
the critical value $\mathrm{Crit}(\alpha;F)$ is equal to $0$ for all classes
$\alpha$ that are represented by constant paths. Note also that
inequality~\eqref{eq:Crit} does \emph{not} hold for critical levels of the
energy functional $E$ since the latter is not additive with respect to
simultaneous concatenation and reparametrization of the paths (see the
discussion in~\cite[\S 10.6]{Goresky-Hingston}). This is main advantage of the
norm functional $F$ over the more popular energy functional $E$.

\smallskip\emph{(3) Symmetries.}

Let
\[
A:\mathcal{P}_{n}\rightarrow\mathcal{P}_{n}, \qquad(A\gamma)(t):=\gamma(1-t)
\]
be the involution given by reversing the direction of paths. This involution
fixes the space of constant paths pointwise. The critical set of $E$ is stable
under $A$ since $E$ is $A$-invariant.

\begin{lemma}
\label{lem:A} For all $\alpha,\beta\in H_{\cdot}(\mathcal{P}_{n})$ we have
\[
A(\alpha\ast\beta)=A(\beta)\ast A(\alpha).
\]

\end{lemma}

\begin{proof}
Let $T$ be the self-diffeomorphism of $\mathcal{P}_{n}\times\mathcal{P}_{n}$
given by $(\gamma,\delta)\mapsto(\delta,\gamma)$. Then $(A\times A)\circ
T=T\circ(A\times A)$ is a self-diffeomorphism of $\mathcal{P}_{n}
\times\mathcal{P}_{n}$ which preserves $\mathcal{C}_{n}$ and satisfies
\[
c\circ\big((A\times A)\circ T\big)=A\circ c,
\]
as well as $\big((A\times A)\circ T\big)\circ s=s\circ\big((A\times A)\circ
T\big)$. The result then follows directly from the definition of the $\ast
$-product, using the naturality of the Alexander-Whitney shuffle product and
the fact that it commutes with $T$ (in the obvious sense).
\end{proof}

\begin{remark}
The space $\mathcal{P}_{n}$ carries yet another involution which we denote
$\gamma\mapsto\bar\gamma$ and call \emph{complex conjugation}. This is induced
by complex conjugation on complex projective space
\[
z=[z_{0}:z_{1}:\dots:z_{n}]\mapsto\bar{z}=[\bar z_{0}:\bar z_{1}:\dots:\bar
z_{n}],
\]
which is an involution that fixes $\mathbb{R}P^{n}$ pointwise.

The involutions $\bar{\quad}$ and $A$ commute with each other. They also
commute with the map
\[
\mathcal{P} _{n}\hookrightarrow\mathcal{P}_{n+1}
\]
induced by any real linear embedding $(\mathbb{C}P^{n},\mathbb{R}
P^{n})\hookrightarrow(\mathbb{C}P^{n+1},\mathbb{R}P^{n+1})$. If $\gamma$ is a
vertical half-circle, then so are $\bar{\gamma}$ and $A\gamma$.
\end{remark}

\smallskip\emph{(4) Heredity.}

As we have already alluded to in the previous paragraph, the embedding
\[
(\mathbb{C}P^{n},\mathbb{R}P^{n})
\hookrightarrow(\mathbb{C}P^{n+1},\mathbb{R}P^{n+1}),\qquad[z_{0}:\dots:z_{n}]\mapsto[z_{0}:\dots:z_{n}:0]
\]
induces an embedding
\[
\mathcal{P}_{n}\hookrightarrow\mathcal{P}_{n+1}.
\]
We denote the induced map in homology
\[
f_{n}:H_{\cdot}(\mathcal{P}_{n})\to H_{\cdot}(\mathcal{P}_{n+1}).
\]
We emphasize that $f_{n}$ is \emph{not} a morphism of rings (not only does it
fail to send the unit into the unit, but it also does not respect the product structure).

However, the map $f_{n}$ is \emph{linear}. The following \emph{Heredity
Principle} is a simple but useful rephrasing of this fact: given a linear
relation between homology classes in $H_{\cdot}(\mathcal{P}_{n})$, their
images under $f_{n}$ must satisfy the \emph{same} linear relation. This
principle is effective since our generators of $H_{\cdot}(\mathcal{P}_{n})$
are geometric, so that their images under $f_{n}$ are easy to identify.

\subsection{Hopf symmetries and section ($n$ odd)} \label{sec:Hopf}

Let $n=2m+1$. The fibers of the Hopf principal bundle
\[
S^{1}\rightarrow S^{n}\rightarrow\mathbb{C}P^{m}
\]
are (oriented) great circles on $S^{n}$. The quotient of the total space by
the subgroup $\mathbb{Z}/2 \subset S^{1}$ is the total space of a principal
$S^{1}/(\mathbb{Z}/2) = \mathbb{R} P^{1}$-bundle
\[
\mathbb{R} P^{1}\rightarrow\mathbb{R} P^{n}\rightarrow\mathbb{C} P^{m}\text{.
}
\]
Because the Hopf fibers are geodesics, and $S^{1}$ acts by isometries, the
fibers of the map $\mathbb{R} P^{n}\rightarrow\mathbb{C} P^{m}$ are (oriented)
geodesics on $\mathbb{R}P^{n}$.

The map that assigns to $x\in S^{n}$ the unit tangent vector $u$ to the Hopf
fiber at $x$ is a global section of the unit tangent bundle $STS^{n}$. If we
view $STS^{n}$ as $\{(x,u)\in\mathbb{R}^{n+1}\times\mathbb{R}^{n+1}\,:\,\Vert
x\Vert=1,\ \Vert u\Vert=1,\,\langle x,u\rangle=0\}$, then this section can be
viewed as the map
\[
x=(x_{0},x_{1},\dots,x_{n-1},x_{n})\mapsto u=(-x_{1},x_{0},\dots
,-x_{n},x_{n-1}).
\]
In complex coordinates this is $u=Jx$, where $J\in S^{1}$ is a fourth root of unity.

If $n\equiv3\ (\operatorname{mod}4)$, then $S^{n}$ can be identified with the
unit sphere in the quaternionic vector space $\mathbb{H}^{(n+1)/4\text{ }}$.
In this case there are three such sections
\[
u=J_{i}x,\qquad i\in\{1,2,3\}
\]
so that for each $x\in$ $S^{n}$ the vectors $Jx=J_{1}x,$ $J_{2}x,$ $J_{3}x\in
T_{x}S^{n}$ form an orthonormal triple.

The map that assigns to $x\in\mathbb{R}P^{n}$ the tangent vector $u$ to the
Hopf fiber at $x$ is a global section of the unit tangent bundle
$ST\mathbb{R}P^{n}$ that descends from the above section of $STS^{n}$ by
viewing $ST\mathbb{R}P^{n}$ as the quotient of $STS^{n}$ by the action of the
derivative of the antipodal map $S^{n}\rightarrow S^{n}$. The expression
$u=Jx$ also makes sense for $x\in\mathbb{R}P^{n}$ and $u\in ST_{x}
\mathbb{R}P^{n}$, and describes the section
\begin{align}
\label{eq:sigma}\sigma:\mathbb{R}P^{n}  &  \rightarrow ST\mathbb{R}P^{n},\\
x  &  \mapsto u=Jx.\nonumber
\end{align}
It follows that the Hopf fiber through $x$ in $\mathbb{R}P^{n}$ is the unique
geodesic with initial vector $u=Jx$. Now let $I$ be the complex structure on
$T\mathbb{C}P^{n}.$ Then
\[
x\mapsto v=IJx
\]
describes a section of the unit \textit{normal} bundle $SN\mathbb{R}P^{n}$ to
$\mathbb{R}P^{n}$ in $\mathbb{C}P^{n}$.

\begin{figure}[ptb]
\begin{center}
\input{IJ.pstex_t}
\end{center}
\caption{Hopf fibers in $S^{2m+1}$.}
\label{fig:IJ}
\end{figure}
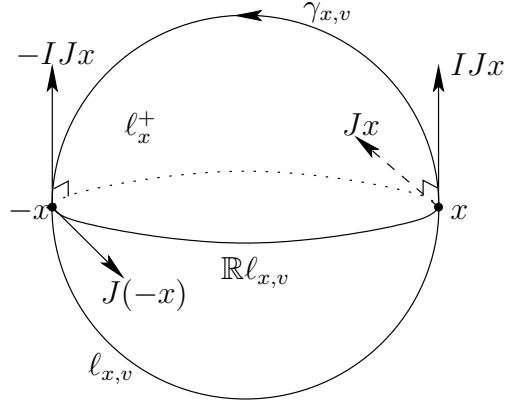

The direction $v=IJx$ at $x\in\mathbb{R}P^{n}\subset\mathbb{C}P^{n}$
determines the complex line $\ell_{x,v}$, whose intersection with
$\mathbb{R}P^{n}$ is thus precisely the Hopf fiber through $x$. Note also that
the equator of $\ell_{x,v}$ inherits an orientation from the Hopf fiber, and
that $\ell_{x,v}$ has a globally defined ``upper hemisphere".

Define
\[
S_{1} :=\{(x,v,\theta)\in Y_{1}:v=IJx\} \ \subseteq\ Y_{1}=SN\mathbb{R}
P^{n}\times S^{1}.
\]
Note that
\[
S_{1}=\{\text{vertical half circles }\gamma:\gamma\subset
\ell_{\gamma(0)}^{+}\}
\]
where $\ell_{\gamma(0)}^{+}$ is the upper half of the complex line containing
the Hopf fiber of $\gamma(0)$. Note also that for each vertical half circle
$\gamma$ we have $\gamma\subset\ell_{\gamma(0)}^{+}$ if and only if
$A\gamma\subset\ell_{\gamma(1)}^{+}$. This gives us the first statement in the
next lemma.

\begin{lemma}
\label{lem:AS1} We have
\begin{align*}
AC(S_{1})  &  =C(S_{1}),\\
AC(Y_{1})  &  =C(Y_{1}),
\end{align*}
where $C:Y_{1}\rightarrow\mathcal{P}_{n}$ is the evaluation map (formerly
denoted $\varphi_{1}$)
\[
(x,v,\theta) \mapsto C_{x,v,\theta}.
\]

\end{lemma}

\begin{proof}
The second statement is a consequence of the fact that if $\gamma$ is a
vertical half circle, so is $A\gamma$. See the example $n=1$ worked out below.
\end{proof}

\subsection{Generators}

Let us first describe the generators of $H_{\cdot}(\mathcal{P}_{n})$ with
respect to the $\ast$-product for $n\geq1$.

\begin{itemize}
\item the ``constant loops class" $U\in H_{n}(\mathcal{P}_{n})$. This is the
image of the fundamental class $[\mathbb{R}P^{n}]$ under the injection
$H_{\cdot}(\mathbb{R}P^{n})\hookrightarrow H_{\cdot}(\mathcal{P}_{n})$. It is
represented by the submanifold $\mathbb{R}P^{n}\subset\mathcal{P}_{n}$
consisting of constant loops.

\item the \textquotedblleft hyperplane class" $H\in H_{n-1}(\mathcal{P}_{n})$.
It is represented by the class of a hyperplane $\mathbb{R}P^{n-1}$ in the
constant loops $\mathbb{R}P^{n}\subset\mathcal{P}_{n}$.

\item the \textquotedblleft completing manifold class" $Y\in H_{2n}
(\mathcal{P}_{n})$. This is represented by the map
\begin{align*}
C  &  :Y_{1}\rightarrow\mathcal{P}_{n},\text{ \ }\\
(x,v,\theta)  &  \mapsto C_{x,v,\theta}.
\end{align*}

\item the \textquotedblleft class of a completed section", $S\in
H_{n+1}(\mathcal{P}_{n})$ for $n=2m+1$ odd. This is represented by the map
\begin{align*}
C  &  :S_{1}\rightarrow\mathcal{P}_{n},\text{ \ }\\
(x,v,\theta)  &  \mapsto C_{x,v,\theta}.
\end{align*}

\item the \textquotedblleft class of a completed section defined along a
hyperplane", $T\in H_{n}(\mathcal{P}_{n})$ for $n=2m$ even. We choose a
hyperplane $\mathbb{R}P^{2m-1}\subset\mathbb{R}P^{2m}$ and consider again the
section $\mathbb{R}P^{2m-1}\rightarrow ST\mathbb{R}P^{2m-1}$ defined above,
viewed as taking values in $ST\mathbb{R}P^{2m}$. The class $T$ is represented
by the restriction of $C$ to
\[
\{(x,v,\theta)\in Y_{1}:x\in\mathbb{R}P^{2m-1}\text{ and }v=IJx\}
\]
(For $n$ even, $J$ is only defined along the hyperplane $\mathbb{R}P^{2m-1}$.)

Alternatively, the class $T\in H_{n}(\mathcal{P}_{n})$, $n=2m$ even is the
image of $S\in H_{n}(\mathcal{P}_{n-1})$ under the map 
$H_{n} (\mathcal{P}_{n-1})\rightarrow H_{n}(\mathcal{P}_{n})$.
\end{itemize}

It follows from Lemma~\ref{lem:AS1} and the Heredity Principle
in~\S \ref{sec:gen}(4) that
\begin{equation}
\label{eq:Aeverything}AS =S,\qquad AY=Y, \qquad AH=H, \qquad AT=T
\end{equation}

We denote $P\in H_{0}(\mathcal{P}_{n})$ the class of a point (taken to be a
constant loop, for example).

We shall sometimes denote the above classes by $U_{n}$, $H_{n}$, $Y_{n}$,
$S_{n}$, resp. $T_{n}$, to indicate that they pertain to the homology of
$\mathcal{P}_{n}$.

\begin{remark}
\label{rmk:n=1} When $n=1$ we have that $H=P$. In this case $ST\mathbb{R}
P^{1}$ is a disjoint union of two circles and the completing manifold is
$Y_{1} =ST\mathbb{R}P^{1}\times S_{\pi}^{1}$, which is a disjoint union of two
tori. For every $x=[x_{0} :x_{1}]\in\mathbb{R}P^{1}$, the vector
$Jx=[-x_{1}:x_{0}]$ lies in $T_{x}\mathbb{R}P^{1}$, and $IJx$ (defined using
the vector $Jx$ and the complex structure $I$ on $\mathbb{C} P^{1}$) is
tangent to $\mathbb{C} P^{1}$ but normal to $\mathbb{R}P^{1}$. The vectors
$IJx$ with $x\in\mathbb{R}P^{1}$ point into the ``upper" hemisphere of
$\mathbb{C} P^{1}$ (the hemisphere about which $\mathbb{R} P^{1}$ has positive
winding number) and the vectors $-IJx$ with $x\in\mathbb{R}P^{1}$ point into
the ``lower" hemisphere. (Complex conjugation reverses the two hemispheres and
fixes the equator $\mathbb{R}P^{1}$.) Thus $S$ \textit{is represented by the
space of vertical half circles in the upper hemisphere with endpoints on
}$\mathbb{R}P^{1}$\textit{, parametrized by the endpoints }$\{(x,x^{\prime
})\in\mathbb{R}P^{1}\times\mathbb{R}P^{1}\}$. The class $Y\in H_{2}
(\mathcal{P}_{1})$ is represented by the sum of $S$ and its complex conjugate
$\overline{S}$ in the lower hemisphere:
\[
Y=S+\overline{S}.
\]

\end{remark}

Using our previous computation of $H_{\cdot}(\mathcal{P}_{n})$, we are now in
a position to write down explicitly generators of all the homology groups. For
simplicity we shall suppress the symbol $\ast$ from the multiplicative
notation, and write for example $HY$ instead of $H\ast Y$. For readability we
refer to a critical level $k\pi/2$ as being \textquotedblleft level
$k$\textquotedblright. We refer to the block $H_{\cdot}(ST\mathbb{R}
P^{n})[-1-(k-1)n]$ in $H_{\cdot}(\mathcal{P}_{n})$ on level $k$ as
\textquotedblleft the block on level $k$". Recall that $H_{\cdot}
(\mathbb{R}P^{n})=\langle U,H,H^{2},\dots,H^{n}\rangle$.

\begin{lemma}
Multiplication on the right by $Y^{s-\ell}$ determines an isomorphism between
the blocks on levels $k=\ell$ and $k=s$ for $s\geq\ell$ and $n\geq2$.

(i) Let $n$ odd $\ge3$. The block on level $1$ is
\[
H_{\cdot}(ST\mathbb{R}P^{n})[-1]\simeq H_{\cdot}(\mathbb{R}P^{n})\ast\langle
S,Y\rangle.
\]

(ii) Let $n$ even $\ge2$. The block on level $1$ is
\[
H_{\cdot}(ST\mathbb{R}P^{n})[-1]\simeq\langle U,H,H^{2},\dots,H^{n-1}
\rangle\ast\langle T,Y\rangle.
\]

(iii) Let $n=1$. The block $H_{\cdot}(ST\mathbb{R}P^{1})[-k]$ on level $k$ is
isomorphic to $H_{\cdot}(\mathbb{R}P^{1})\ast\langle S\bar S S\dots,\bar S S
\bar S\dots\rangle$, where each of the products $\dots S\bar S S\dots$ has $k$ factors.
\end{lemma}

\begin{proof}
To prove statement~(i), let us recall the computation of $H_{\cdot
}(ST\mathbb{R}P^{n};\mathbb{Z}/2)$ in~\S \ref{sec:homologyPn}. A basis of
$H_{\cdot}(ST\mathbb{R}P^{n})$ is obtained by restricting the section $\sigma$
in~\eqref{eq:sigma}, respectively the bundle $ST\mathbb{R}P^{n}$, to the
linear subspaces $\mathbb{R}P^{k}$, $0\le k\le n$. The resulting homology
classes can alternatively be described as the intersection products between
the classes $[\sigma]\in H_{n}(\mathbb{R}P^{n})$, resp. $[ST\mathbb{R}
P^{n}]\in H_{2n-1}(\mathbb{R}P^{n})$ with the classes $[\mathbb{R}P^{k}]$,
$0\le k\le n$. Note in particular that $H_{\cdot}(ST\mathbb{R}P^{n})$ has a
basis consisting of classes that are represented by submanifolds.

To understand the block $H_{\cdot}(ST\mathbb{R}P^{n})[-1]$ on level $1$,
recall that $Y_{1}$ is a (trivial) $S^{1}_{\pi}$-bundle over $K_{1}\equiv
ST\mathbb{R}P^{n}$ by $(x,v,\theta)\overset{\pi}{\mapsto}(x,v)$. Each cycle in
$ST\mathbb{R}P^{n}$ gives rise to a cycle in $\mathcal{P}_{n}$ by taking its
preimage under the bundle map $\pi$ and evaluating into $\mathcal{P}_{n}$. In
view of the above description of $H_{\cdot}(ST\mathbb{R}P^{n})$ and in view of
the definition of the classes $S$ and $Y$, we obtain $H_{\cdot}(ST\mathbb{R}
P^{n})[-1]\simeq H_{\cdot}(\mathbb{R}P^{n})\ast\langle S,Y\rangle$.

The proof of statement~(ii) is similar to the proof of statement~(i).

The statement concerning right multiplication by $Y^{s-\ell}$ follows from
Lemma~\ref{lem:geom*}. To see this, we recall the identification
$SN\mathbb{R}P^{n}\equiv ST\mathbb{R}P^{n}$, the bundle maps $p_{k}
:Y_{k}\rightarrow ST\mathbb{R} P^{n}$, $k\geq1$ described in
Remark~\ref{rmk:pk}, and the evaluation maps $\varphi_{k}:Y_{k}\rightarrow
\mathcal{P}_{n}$, $k\geq1$ in~\S \ref{sec:homologyPn}. Also, given a
submanifold $X\subset Y_{k}$ we denote $[X]\in H_{\cdot}(\mathcal{P}_{n})$ the
class represented by the map $\varphi_{k}|_{X}$.

Let now $A\subset ST\mathbb{R}P^{n}$ be a submanifold and denote
$A_{k}:=(p_{k})^{-1}(A)\subset Y_{k}$, $k\geq1$. Then
\[
A_{s}=A_{\ell}\ {_{{\mathrm{ev}_{1}}\circ f_{\ell}}}\!\times
{_{{\mathrm{ev}_{0}}\circ f_{s-\ell}}}\ Y_{s-\ell}
\]
by definition of the manifolds $Y_{k}$, $k\geq1$. By associativity of the map
$c_{\min}$ we have $c_{\min}\circ(\varphi_{\ell},\varphi_{s-\ell})=\varphi
_{s}$, and Lemma~\ref{lem:geom*} implies
\[
\lbrack A_{s}]=[A_{\ell}]\ast\lbrack Y_{s-\ell}].
\]
By~\eqref{eq:YkY1k} we obtain $[A_{s}]=[A_{\ell}]\ast Y^{s-\ell}$. Since for
each $k\geq1$ the classes $[A_{k}]$ form a basis of the block on level $k$ as
$A$ ranges through a collection of submanifolds of $ST\mathbb{R}P^{n}$ which
represent a basis of $H_{\cdot}(ST\mathbb{R}P^{n})$, the conclusion follows.

Statement~(iii) follows from the above and from the discussion in
Remark~\ref{rmk:n=1}.
\end{proof}

The resulting explicit description of a set of generators for the homology
groups $H_{\cdot}(\mathcal{P}_{n})$ is the following.

\begin{itemize}
\item \textbf{$\mathbf{n}$ even $\mathbf{\ge2}$}. The block $H_{\cdot
}(\mathbb{R}P^{n})$ on level $0$ is generated by $U,H,H^{2},\dots,\break
H^{n}=P$. The block $H_{\cdot}(ST\mathbb{R}P^{n})[-1-(k-1)n]$ on level $k$ is
generated by $TY^{k-1},HTY^{k-1},\dots,H^{n-1}TY^{k-1}$ in (descending)
degrees $kn,\dots,1+(k-1)n$, and $Y^{k},HY^{k},\dots, H^{n-1}Y^{k}$ in
(descending) degrees $(k+1)n,\dots, 1+kn$.

\item \textbf{$\mathbf{n}$ odd $\mathbf{\ge3}$}. The block $H_{\cdot
}(\mathbb{R}P^{n})$ on level $0$ is generated by $U,H,H^{2},\dots,\break
H^{n}=P$. The block $H_{\cdot}(ST\mathbb{R}P^{n})[-1-(k-1)n]$ on level $k$ has
generators $SY^{k-1},HSY^{k-1},H^{2}SY^{k-1},\dots,H^{n}SY^{k-1}$ in
(descending) degrees $1+kn,\dots, 1+(k-1)n$, and $Y^{k},HY^{k},H^{2}
Y^{k},\dots, H^{n}Y^{k}$ in descending degrees $(k+1)n, \dots, kn$.

\item $\mathbf{n=1}$. The block $H_{\cdot}(\mathbb{R}P^{1})$ on level $0$ is
generated by $U$ (in degree $1$) and $H=P$ (in degree $0$). The block
$H_{\cdot}(ST\mathbb{R}P^{1})[-k]=H_{\cdot}(S^{1}\,\sqcup\,S^{1})[-k]$ on
level $k$ is generated by $S\overline{S}S\dots$ and $\overline{S}S\overline
{S}...$($k$ factors each) in degree $k+1$, and by $HS\overline{S}S\dots$and
$H\overline{S}S\overline{S}...$in degree $k$. (Note that according to our
conventions, $Y=S+\overline{S}$.)
\end{itemize}

It is useful to depict the first levels for the cases $n=1,2,3,4$ in a common table.

\begin{center}
{\scriptsize
\begin{tabular}
[c]{c|||c|c|c|c|||l|l|l|||l|l|l|||l|l|l}
$9$ &  &  &  &  &  &  &  &  &  &  &  &  & \\
$8$ &  &  &  &  &  &  &  &  &  &  &  & \qquad\quad\quad\ \ $Y$ & \\
$7$ &  &  &  &  &  &  &  &  &  &  &  & \qquad\quad\, \ $HY$ & \\
$6$ &  &  &  &  &  &  &  &  & \qquad\qquad\ \ $Y$ &  &  & \qquad\quad\,
$H^{2}Y$ & \\
$5$ &  &  &  &  &  &  &  &  & \qquad\quad\, \ $HY$ &  &  & \qquad\quad\,
$H^{3}Y$ & \\
$4$ &  &  &  &  &  & \qquad\quad\ \ \!$Y$ &  &  & $S$ \qquad\ $H^{2}Y$ &  &
$U_{4}$ & $T$ & \\
$3$ &  &  & $S\overline{S},\overline{S}S$ &  &  & \qquad\ \ $HY$ &  & $U_{3}$
& $HS$ \quad\ $H^{3}Y$ &  & $H$ & $HT$ & \\
$2$ &  & $S,\overline{S}$ & $HS\overline{S},H\overline{S}S$ &  & $U_{2}$ & $T$
&  & $H$ & $H^{2}S$ &  & $H^{2}$ & $H^{2}T$ & \\
$1$ & $U_{1}$ & $HS,$ $H\overline{S}$ &  &  & $H$ & $HT$ &  & $H^{2}$ &
$H^{3}S$ &  & $H^{3}$ & $H^{3}T$ & \\
$0$ & $H$ &  &  &  & $H^{2}$ &  &  & $H^{3}$ &  &  & $H^{4}$ &  & \\\hline
& $0$ & $\frac{\pi}{2}$ & $\pi$ & \dots & \ $0$ & \qquad$\frac{\pi}{2}$ &
\dots & \ $0$ & \qquad\ $\frac{\pi}{2}$ & \dots & \ $0$ & \qquad\ $\frac{\pi
}{2}$ & \dots\\
\quad &  &  &  &  &  &  &  &  &  &  &  &  & \\
&  & $n=1$ &  &  &  & \quad$n=2$ &  &  & \quad\ \ $n=3$ &  &  & \quad\, $n=4$
&
\end{tabular}
}
\end{center}

\begin{remark}
\label{rmk:heredity} The maps $\mathcal{P}_{n}\rightarrow
\mathcal{P}_{n+1}$, $n\geq2$ induce injections in homology on the first two
columns of the above table. In particular
\[
H^{k}\rightarrow H^{k+1},H^{k}T\rightarrow H^{k+2}S,H^{k}S\rightarrow H^{k}T.
\]
Note that the degree of the generators on the other columns goes uniformly to $\infty$ as $n\to\infty$. This yields in particular the computation of the homology of the space 
$$
\mathcal P_\infty := \mathcal P_{\mathbb{R}P^\infty}\mathbb{C}P^\infty:=\lim_{\stackrel \longrightarrow n}Ê\mathcal P_n.
$$
The homology $H_\cdot(\mathcal P_\infty)$ is the limit of the first two columns in the above table as $n\to \infty$, it has rank $1$ in degree $0$ and rank $2$ in all the other degrees. 
\end{remark}

\subsection{Relations}

At this point we have already proved half of our main theorem, since we have
seen that $1,H,S,Y$, respectively $1,H,T,Y$ generate 
$H_{\cdot}(\mathcal{P}_{n})$ as a ring. We now have to identify the relations satisfied by these generators.

\subsubsection{\textbf{The case $n=1$.}}

\begin{lemma}
\label{lem:1} The following relations hold for $n=1$:
\[
PS+SP=P\overline{S}+\overline{S}P=U,\qquad S^{2}=0,\qquad\overline{S}^{2}=0.
\]

\end{lemma}

\begin{proof}
To prove the relation $S^{2}=0$ we note that $S^{2}$ is represented by a cycle
in $\mathcal{P}_{1}^{\leq\pi}$ and this cycle does not contain any geodesics
of length $\pi$. Hence $\mathrm{Crit}(S^{2})<\pi$. On the other hand $S^{2}$
lives in degree $3$, but $H_{3}($ $\mathcal{P}_{1}^{<\pi})=0$ since $H_{\cdot
}($ $\mathcal{P}_{1}^{<\pi})$ is generated by $P$ and $U_{1}$. Thus $S^{2}$
must vanish. The relation $\overline{S}^{2}=0$ is proved similarly.

We now prove $PS+SP=U$. The Hopf fibration determines an orientation on
$\mathbb{R}P^{1},$ and at any point $x\in\mathbb{R}P^{1}$, the vector $v=IJx$
points into the ``upper hemisphere" $\ell^{+}$ of $\mathbb{C}P^{1}$. Then $S$
can be alternatively described as the class of the cycle $s:\mathbb{R}
P^{1}\times\mathbb{R}P^{1}\rightarrow\mathcal{P}_{1}$ that associates to a
pair of points $(x,x^{\prime})$ the unique vertical half-circle from $x$ to
$x^{\prime}$ and contained in $\ell^{+}$. The cycle $PS$ is represented by the
restriction of $s$ to $\{\ast\}\times\mathbb{R}P^{1}$, and $SP$ is represented
by the restriction of $s$ to $\mathbb{R}P^{1}\times\{\ast\}$. If we represent
the torus $\mathbb{R}P^{1}\times\mathbb{R}P^{1}$ as a square with opposite
sides identified, the cycles $PS$ and $SP$ correspond to two adjacent sides.
(See Figure~\ref{fig:torus}.) The cycle $U=U_{1}$ in turn corresponds to the
diagonal, hence $PS+SP=U$.

The relation $P\overline{S}+\overline{S}P=U$ can be proved in the same way.
\begin{figure}[ptb]
\begin{center}
\input{torus.pstex_t}
\end{center}
\caption{The space of half-circles contained in the upper hemisphere on
$\mathbb{C}P^{1}$ is modeled on $\mathbb{R}P^{1}\times\mathbb{R}P^{1}$ using
the endpoints.}
\label{fig:torus}
\end{figure}
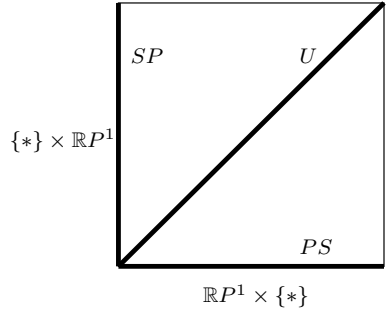
\end{proof}

\begin{proof}
[Proof of Theorem~\ref{thm:PCSproduct} in the case $n=1$]We have already
noticed above that
\[
Y=S+\overline{S}, \qquad H=P.
\]
The homology of $\mathcal{P}_{1}$ on level $\pi/2$ can alternatively be
described as
\[
H_{\cdot}(ST\mathbb{R}P^{1})[-1]=\langle HS,S, HY,Y\rangle,
\]
and more generally the homology of $\mathcal{P}_{1}$ on level $k\pi/2$ can be
described as
\[
H_{\cdot}(ST\mathbb{R}P^{1})[-k]=\langle HS,S,HY,Y\rangle\ast Y^{k-1} =
\langle HSY^{k-1}, SY^{k-1},HY^{k},Y^{k}\rangle.
\]
The relations in Lemma~\eqref{lem:1} imply
\[
HS+SH=1,\qquad HY+YH=0,\qquad S^{2}=0,
\]
as well as
\[
YS+SY=Y^{2}.
\]
Finally, we have $H^{2}=0$ since $H$ is the class of a constant path.
\end{proof}

\subsubsection{\textbf{The case $n\ge3$ odd.}}

\begin{lemma}
\label{lem:HSSHU} Let $n\ge3$ odd. The following relation holds
\[
HS+SH=U.
\]

\end{lemma}

\begin{proof}
For index reasons, $SH$ is a linear combination of $U,HS,H^{n}Y$. We first
show that $SH$ is a linear combination of $U$ and $HS$. The image of the
section $\sigma:\mathbb{R}P^{n}\hookrightarrow ST\mathbb{R}P^{n}=K_{1}$ is a
closed set supporting homology classes whose image under the Thom isomorphism
$\tau:H_{\cdot}(K_{1})\rightarrow H_{\cdot+1}(\mathcal{P}_{n}^{\leq\frac{\pi
}{2}},\mathcal{P}_{n}^{<\frac{\pi}{2}})$ contains $HS$ and $SH$ (both of which
are represented by ``subsets of $S$") but not supporting $H^{n}Y$, since the
image of $\sigma$ is diffeomorphic to $\mathbb{R}P^{n}$, and $\mathrm{rk}
\,H_{n}(\mathbb{R}P^{n})=1$; but $HS$ and $H^{n}Y$ span a subspace of the
homology of rank 2. By Corollary~\ref{cor:crit2} the class $SH$ is a multiple
of $HS$ mod $H_{\cdot}(\mathcal{P}_{n}^{<\frac{\pi}{2}})$.

We also have:

\begin{itemize}
\item $SH\neq0$ by Lemma~\ref{lem:A} (since $HS\neq0$) .

\item $SH\neq HS$ (otherwise we would obtain $SH^{n}=H^{n}S$, whereas
Lemma~\ref{lem:1} and Remark~\ref{rmk:heredity} on heredity prove that
$SH^{n}+H^{n}S=H^{n-1}$).

\item $SH\neq U$ (otherwise $U=A(U)=A(SH)=HS$).
\end{itemize}

Thus $SH=HS+U$, which is equivalent to $SH+HS=U$.
\end{proof}

\begin{lemma}
\label{lem:SSbar} Let $n\geq1$ odd. The following relation holds
\[
S+\overline{S}=\left\{
\begin{array}
[c]{cl}
0, & \mbox{if } n\equiv3 \ (\operatorname{mod}4),\\
H^{n-1}Y, & \mbox{if } n\equiv1 \ (\operatorname{mod}4).
\end{array}
\right.
\]

\end{lemma}

\begin{proof}
We have already seen that this is true when $n=1$. By the second assertion in
Lemma~\ref{lem:completing} and because there is no homology in 
$\mathcal{P}_{n}$ in the degree of $S$ below level $\frac{\pi}{2}$,\textbf{ } it is enough
to prove the corresponding relation (under the Thom isomorphism) in the
homology of $SN\mathbb{R}P^{n}\simeq ST\mathbb{R}P^{n}$; let us denote by
$h^{k}s$ and $h^{k}y$ the generators in $H_{\cdot}(ST\mathbb{R}P^{n})$
corresponding to $H^{k}S$ and $H^{k}Y$. They are represented by
\begin{align*}
h^{k}s_{1}  &  :=\{(x,u)\in ST\mathbb{R}P^{n}:x\in\mathbb{R}P^{n-k}\text{ and
}u=Jx\}\\
h^{k}y_{1}  &  :=\{(x,u)\in ST\mathbb{R}P^{n}:x\in\mathbb{R}P^{n-k}\}
\end{align*}
The involution on $ST\mathbb{R}P^{n}$ corresponding to the involution
$\gamma\rightarrow\overline{\gamma}$ on $\mathcal{P}_{n}$ is the map
$(x,u)\rightarrow(x,-u)$. Thus the generator $\overline{s}$ in $H_{\cdot
}(ST\mathbb{R}P^{n})$ corresponding to $\overline{S}$ is represented by
\[
\overline{s}_{1}:=\{(x,u)\in ST\mathbb{R}P^{n}:x\in\mathbb{R}P^{n}\text{ and
}u=-Jx\}
\]
The group $H_{n}(ST\mathbb{R}P^{n})$ is generated by the elements $s$ and $h^{n-1}y$. The
dual group $H_{n-1}(ST\mathbb{R}P^{n})$ is generated by $hs$ and $h^{n}y$. The
reader can verify that
\begin{align*}
h^{n-1}y\cap h^{n}y  &  =0,\\
h^{n-1}y\cap hs  &  =\ast,\\
s\cap h^{n}y =\overline{s}\cap h^{n}y  &  =\ast,\\
\overline{s}\cap hs  &  =0.
\end{align*}
where $(\ast)$ is the generator of $H_{0}(ST\mathbb{R}P^{n})$. It remains to
compute $s\cap hs$.

Let $n\equiv3 \ (\operatorname{mod}4)$. In this case the one-parameter family
of sections $S^{n}\to ST\mathbb{S}^{n}$ given by $x\mapsto(x,u_{t})$ with
\[
u_{t}:=(\cos tJ_{1}+\sin tJ_{2})x
\]
descends to a family of sections $\sigma_{t}$: $\mathbb{R}P^{n}\rightarrow
ST\mathbb{R}P^{n}$ that represent $s$ but do not intersect $s_{1}$ for $t>0$.
Thus
\[
s\cap hs=0,
\]
so that $s$ and $\overline{s}$ have the same intersection numbers and are
therefore homologous.

Let now $n\equiv1 \ (\operatorname{mod}4)$. We claim that
\begin{equation}
\label{eq:shsast}s\cap hs=\ast,
\end{equation}
so that the classes $s$ and $\overline{s}+h^{n-1}y$ have the same intersection
numbers and are therefore equal.

The sphere $S^{n}$ is now the unit sphere in $\mathbb{H}^{(n-1)/4}
\oplus\mathbb{C}$. On $\mathbb{H}^{(n-1)/4}$ we have the complex structures
$J_{1}$ and $J_{2}$, while on $\mathbb{C}$ we have the complex structure
$J=J_{1}$.

\medskip

\noindent\textbf{Notation.} Given $x\in\mathbb{R}^{n+1}$, we write
\[
x=:(y,z)=:(y(x),z(x))
\]
where $y=(y_{0},...,y_{n-2})=(x_{0},...,x_{n-2})\in\mathbb{R}^{n-1}
\simeq\mathbb{H}^{(n-1)/4}$ and $z=(z_{0},z_{1})=(x_{n-1},x_{n})\in
\mathbb{R}^{2}\simeq\mathbb{C}$.

\medskip

The one-parameter family of sections $S^{n}\rightarrow ST\mathbb{S}^{n}$ given
by $(y,z)\rightarrow((y,z),u_{t})$ with
\[
u_{t}:=((\cos(t)J_{1}+\sin(t)J_{2})y,Jz)
\]
descends to a family of sections $\sigma_{t}:\mathbb{R}P^{n}\rightarrow
ST\mathbb{R}P^{n}$. Each of these sections represents $s$ and, if $t>0$,
$\sigma_{t}$ agrees with $\sigma=\sigma_{0}$ precisely when $y=0$, i.e. when
$x\in\mathbb{R} P^{1}=\{[0:0:...:0:y_{n-1}:y_{n}]\}$. In particular, each
section $\sigma_{t}$ intersects $hs_{1}=\{[x,u]\in ST\mathbb{R}P^{n}:y_{n}=0
\mbox{ and } u=Jx\}$ in exactly one point, namely the image in $ST\mathbb{R}
P^{n}$ of the point $(x_{\ast},u_{\ast})$ with
\begin{align*}
x_{\ast}  &  =(0,0,....,0,1,0),\\
u_{\ast}  &  =Jx=(0,0,....,0,1).
\end{align*}
To prove~\eqref{eq:shsast} it is enough to show that this intersection is
transverse for some fixed $t>0$. This is a local question for which we need to
verify that the tangent space $T_{(x_{\ast},u_{\ast})} ST\mathbb{S}^{n}$ is
spanned by the tangent space at $(x_{\ast},u_{\ast})$ to the submanifolds
\[
\mathcal{V}:=\{(x,u)\in ST\mathbb{S}^{n}:x\in S^{n-1}=\{y_{n}=0\} \mbox{ and }
u=Jx\}
\]
(whose image in $ST\mathbb{R}P^{n}$ is $hs_{1}$) and
\[
\mathcal{W}:=\{(x,u_{t} (x)):x\in\mathbb{S}^{n}\}
\]
(whose image represents $s$).

We view $ST\mathbb{S}^{n}$ as
\[
ST\mathbb{S}^{n}=\{(x,u)\in\mathbb{R}^{n+1}\times\mathbb{R}^{n+1}\, : \,
\|x\|=1,\ \|u\|=1,\ \langle x,u\rangle=0\},
\]
so that
\[
T_{(x,u)}ST\mathbb{S}^{n}=\{(\xi,\eta)\in\mathbb{R}^{n+1}\times
\mathbb{R}^{n+1}\, : \, \langle\xi,x\rangle= 0,\ \langle\eta,u\rangle=0,\ \langle
\xi,u\rangle+ \langle x,\eta\rangle=0\}.
\]
Thus
\[
T_{(x_{\ast},u_{\ast})}ST\mathbb{S}^{n}=\{(\xi,\eta)\, : \, \xi_{n-1}
=0,\ \eta_{n}=0, \ \xi_{n}+\eta_{n-1}=0\},
\]
whereas
\begin{align*}
T_{(x_{\ast},u_{\ast})}\mathcal{V }  &  = \{(\xi,\eta)\in 
T_{(x_{\ast},u_{\ast})}ST\mathbb{S}^{n}\, : \, \xi_{n}=0, \ \eta=J\xi\}\\
&  = \{(\xi,\eta)\, : \, z(\xi)=z(\eta)=0, \ y(\eta)=J_{1}y(\xi)\}
\end{align*}
and
\begin{align*}
T_{(x_{\ast},u_{\ast})}\mathcal{W }  &  = \{(\xi,\eta)\in 
T_{(x_{\ast},u_{\ast})}ST\mathbb{S}^{n}\, : \, y(\eta)=(\cos t J_{1}+ \sin t J_{2})
y(\xi)\}.
\end{align*}
For $t\notin\pi\mathbb{Z}$ we see that the intersection of the last two
subspaces is $\{0\}$, hence transverse because they have complementary dimensions.
\end{proof}

\begin{lemma}
\label{lem:YSSY} Let $n\geq3$ odd. The following relation holds
\[
YS+SY=\left\{
\begin{array}
[c]{cl}
0, & \mbox{if } n\equiv3 \ (\operatorname{mod}4),\\
H^{n-1}Y^{2}, & \mbox{if } n\equiv1 \ (\operatorname{mod}4).
\end{array}
\right.
\]
\end{lemma}

\begin{proof}It is enough to prove that 
\begin{equation}
YS=\overline{S}Y,  \label{eq:YSSbarY}
\end{equation}
from which the result follows in view of Lemma~\ref{lem:SSbar}. In turn,
it is enough to prove the relation that corresponds to~\eqref{eq:YSSbarY}
under the Thom isomorphism in the homology of $SN\mathbb{R}P^{n}\simeq ST
\mathbb{R}P^{n}$. This reduction is made possible by the second assertion in
Lemma~\ref{lem:completing} because there is no homology in $\mathcal{P}_{n}$
in the degree of $YS$ and $\overline{S}Y$ below level $\pi$  (compare with the
proof of Lemma~\ref{lem:SSbar}).

Recall the Thom isomorphism $H_{\cdot }(Y_{2},Y_{2}\setminus L_{2})\simeq
H_{\cdot }(L_{2})$ from Remark~\ref{rmk:perfect}. The classes $YS$ and $
\overline{S}Y$ are represented by spaces of concatenations $c(\gamma ,\delta
)$ of vertical half-circles in the image of $Y_{2}$ as follows. Recall from~\S\ref{sec:Hopf} the definition of the upper hemisphere $\ell_x^+$ of the complex line $\ell_x$ containing the Hopf fiber through $x\in\mathbb{R}P^n$, and denote $\ell_x^-:=\overline{\ell_x^+}$ the corresponding lower hemisphere. For $YS$ a representing submanifold is 
\[
\mathcal{V}:=\{c(\gamma,\delta) \, : \, \delta \subset \ell_{\delta (0)}^{+}=\ell_{\gamma(1)}^{+}\},
\]
whereas for $\overline{S}Y$ a representing submanifold is 
\[
\mathcal{W}:=\{c(\gamma,\delta) \, : \, \gamma \subset \ell_{\gamma
(1)}^{-}=\ell_{\delta (0)}^{-}\}.
\]
Note that $\mathcal{V}$ and $\mathcal{W}$ contain the same geodesics, since
for a smooth geodesic $c(\gamma,\delta)$, the first half $\gamma$ is in
the lower hemisphere $\ell _{\gamma (1)}^{-}$ if and only if the second half 
$\delta$ is in the upper hemisphere $\ell_{\delta (0)}^{+}$. Thus the
statement~\eqref{eq:YSSbarY} will follow from 
Lemma~\ref{lem:completing-local}(b) once we establish transversality. But both
submanifolds $\mathcal{V}$ and $\mathcal{W}$ have codimension $n-1$ and
intersect with $L_{2}$ transversally. Indeed, given a geodesic 
$c(\gamma,\delta)\in\mathcal{V}\cap L_{2}$, the tangent vectors to $L_{2}$
that correspond to variations of the tangent vector 
$\dot\gamma(1)=\dot\delta(0)$ inside $ST_{\gamma_{p}(\frac{\pi}{2})}
\mathbb{R}P^{n}$ while keeping the midpoint of the geodesic fixed, span a
subspace of dimension $n-1$ that is complementary to 
$T_{c(\gamma,\delta)}\mathcal{V}$. 
The same argument applies to $\mathcal{W}$.
\end{proof}

\begin{lemma}
\label{lem:S2=0} Let $n\geq3$ odd. The following relation holds
\[
S^{2}=0.
\]

\end{lemma}

\begin{proof}
Denote $I:=[0,1]$ and let
\[
\widetilde{S_{1}}=\{\gamma:(I,\partial I)\rightarrow
(\mathbb{C}P^{n},\mathbb{R}P^{n}):\gamma\subset\ell_{\gamma(0)}^{+}\}.
\]
There is a strong deformation retraction $\widetilde{S_{1}}\rightarrow S_{1}$
that preserves the space of paths in $\widetilde{S_{1}}$ beginning at $x$ and
ending at $x^{\prime}$ for each pair $(x,x^{\prime})$ with 
$x\in\mathbb{R}P^{n}$ and $x^{\prime}\in\mathbb{R}\ell_{x}$ (the Hopf fiber of $x$). This can be accomplished as follows: there
is a natural continuous family of diffeomorphisms parameterized by 
$x\in\mathbb{R}P^{n}:$
\[
\chi_{x}:D^{2}\rightarrow\ell_{x}^{+}
\]
where $D^{2}$ is a disk of radius $1$ with the flat metric. If $x$ and
$x^{\prime}$ belong to the same Hopf fiber then $\chi_{x}^{-1}\chi_{x^{\prime
}}$ is an isometry of $D^{2}$. We then use affine interpolation, with the
inherited affine structure from $D^{2},$ between an arbitrary path
$\gamma:[0,1]\rightarrow\ell_{\gamma(0)}^{+}$ and the vertical half circle in
$\ell_{\gamma(0)}^{+}$ with the same endpoints. (This works because $D^{2}$ is
convex.) Thus the inclusion $S_{1}\hookrightarrow\widetilde{S_{1}}$ is a
homotopy equivalence. In particular $H_{\cdot}(\widetilde{S_{1}})=0$ in
degrees greater than $\dim S_{1}=n+1$. On the other hand, in view of
Lemma~\ref{lem:geom*} the degree $n+2$ homology class $S^{2}$ is represented
by
\[
S_{1}\,{_{\mathrm{ev}_{1}}}\!\!\times_{\mathrm{ev}_{0}}S_{1}:=\{\gamma
\cdot\delta:\gamma,\delta\in S_{1}\mbox{ and }\gamma(1)=\delta(0)\}
\]
since the self intersection is transverse over the evaluation map. (Here we
identify $S_{1}$ with its image in $\mathcal{P}_{n}$.) Clearly 
$S_{1}\,{_{\mathrm{ev}_{1}}}\!\!\times_{\mathrm{ev}_{0}}S_{1}\subset\widetilde
{S_{1}}$. Thus the class $S^{2}\in H_{n+2}(\mathcal{P}_{n})$ is in the image
of $H_{n+2}(\widetilde{S_{1}})$ and therefore vanishes.
\end{proof}

\begin{lemma}
\label{lem:HYYH} For all $n\ge2$ we have
\[
HY=YH.
\]

\end{lemma}

\begin{proof}
Since $A(YH)=HY$ we obtain $YH\neq0$. Now $\mathrm{Crit}(YH)\le\mathrm{Crit}
(Y)+\mathrm{Crit}(H)=\pi/2$. On the other hand $YH$ lives in degree $2n-1$,
and the only non-zero class in this degree and critical level $\le\pi/2$ is
$HY$. Thus $YH=HY$.
\end{proof}

\emph{This proves Theorem~\ref{thm:PCSproduct} in the case $n$ odd.}

\subsubsection{\textbf{The case $n\ge2$ even.}}

\begin{lemma}
Let $n\ge2$ even. The following relation holds
\[
TY+YT=Y.
\]

\end{lemma}

\begin{proof}
For index reasons we have $YT=\alpha TY+\beta Y$.

Recall from~\S \ref{sec:gen}(4) the inclusion $\mathcal{P}_{n}\hookrightarrow
\mathcal{P}_{n+1}$ determined by a real hyperplane embedding $(\mathbb{C}
P^{n},\mathbb{R} P^{n})\hookrightarrow(\mathbb{C}P^{n+1},\mathbb{R}P^{n+1})$
and the induced map in homology $f_{n}:H_{\cdot}(\mathcal{P}_{n})\to H_{\cdot
}(\mathcal{P}_{n+1})$. It follows from the geometric description of our cycles
that
\[
f_{n}T=H^{2}S, \qquad f_{n}Y=HYH=H^{2}Y,
\]
\[
f_{n}(TY)=H^{2}SYH=H^{2}SHY=H^{3}SY+H^{2}Y,
\]
\[
f_{n}(YT)=HYH^{2}S=H^{3}YS= \left\{
\begin{array}
[c]{ll}
H^{3}SY, & \! \! n+1\equiv3 \ (\mathrm{mod}\ 4)\\
H^{3} SY + H^{n+2} Y, & \! \! n+1 \equiv1\ (\mathrm{mod}\ 4)
\end{array}
\right.  \! = H^{3}SY.
\]
Thus $H^{3}SY=f_{n}(YT)=\alpha f_{n}(TY)+\beta f_{n}Y=\alpha H^{3}
SY+(\alpha+\beta) H^{2}Y$. Since $H^{3}SY$ and $H^{2}Y$ are linearly
independent, this implies $\alpha=1$ and $\beta=0$.
\end{proof}

\begin{lemma}
Let $n\ge2$ even. The following relation holds
\[
TH+HT=H.
\]

\end{lemma}

\begin{proof}
Consider the map $f_{n-1}:H_{\cdot}(\mathcal{P}_{n-1})\to H_{\cdot
}(\mathcal{P}_{n})$ induced in homology by the real inclusion 
$\mathcal{P}_{n-1}\hookrightarrow\mathcal{P}_{n}$. We have
\[
f_{n-1}S=K, \qquad f_{n-1}(HS)=HT,\qquad f_{n-1}U=H.
\]
We have $f_{n-1}\circ A=A\circ f_{n-1}$ and, in view of~\eqref{eq:Aeverything}
and Lemma~\ref{lem:A}, we also have $A(HS)=SH$ and $A(HT)=TH$. We thus obtain
\[
f_{n-1}(SH)=TH.
\]
Now the \emph{linear} relation $SH+HS=U$ implies $TH+HT=H$. (This is an
instance of the \emph{Heredity principle} in~\S \ref{sec:gen}(4).)
\end{proof}

\begin{lemma}
\label{lem:T2T} Let $n\ge2$ even. The following relation holds
\[
T^{2}=T.
\]

\end{lemma}

\begin{proof}
For readability we divide the proof in several steps.

\noindent\emph{Preliminaries.} \smallskip

We recall the map $C:Y_{1}\to\mathcal{P}_{n}$, $(x,v,\theta)\mapsto
C_{x,v,\theta}$. In order to simplify the notation we write in the sequel
$\gamma$ instead of $C(\gamma)$, with $\gamma=(x,v,\theta)\in\mathcal{P}_{n}$.
Whereas $\gamma$ refers to an element of $Y_{1}$ or to the path $C(\gamma)$
will be clear from the context. We also introduce the notation $\mathbb{P}(V)$
for the projective space associated to a real vector space $V$. We write
$n=2t$.

We also recall that, given an element $x\in\mathbb{R}P^{2t-1}$, we denote
$\ell^{+}_{x}$ the ``upper hemisphere'' of the complex line $\ell_{x,Jx}$,
i.e. the hemisphere towards which points the vector $IJx\in T_{x}
\mathbb{C}P^{2t-1}$. The boundary $\ell^{+}_{x}$ is the real projective line
through $x$ and tangent to $Jx$, i.e. the fiber of the Hopf fibration
$\mathbb{R}P^{2t-1}\to\mathbb{C}P^{t-1}$ obtained by factorizing the standard
Hopf fibration $\mathbb{S}^{2t-1}\overset{2:1}{\longrightarrow}\mathbb{R}
P^{2t-1}\longrightarrow\mathbb{C}^{t-1}$. We denote this Hopf fiber by
$h_{x}:=\partial\ell^{+}_{x}$, or $h_{x}^{\mathbb{R}P^{2n-1}}$ in order to
emphasize the total space of the Hopf fibration.

\smallskip

\noindent\emph{Description of $T$.} \smallskip

Recall that $T$ is the image of $S$ under the map $f_{n-1}:H_{n}
(\mathcal{P}_{n-1})\rightarrow H_{n}(\mathcal{P}_{n})$ induced by the
embedding
\[
(\mathbb{C}P^{n-1}, \mathbb{R}P^{n-1})\hookrightarrow(\mathbb{C}
P^{n},\mathbb{R}P^{n})
\]
given by $[z_{0}:\dots:z_{n-1}]\mapsto[z_{0}:\dots:z_{n-1}:0]$. Thus $T$ is
represented by
\begin{align*}
T_{1}  &  :=\{\gamma\in Y_{1}:\gamma(0)\in\mathbb{R} P^{n-1}\text{ and }
\mathrm{im}(\gamma)\subset\ell_{\gamma(0)}^{+}\}\\
&  \ =\{\gamma\in Y_{1}:\gamma(1)\in\mathbb{R} P^{n-1}\text{ and }
\mathrm{im}(\gamma)\subset\ell_{\gamma(1)}^{+}\}.
\end{align*}

For further use let us write $\mathbb{R}P^{n}=
\mathbb{P}(\mathbb{R}^{2t-2}\times\mathbb{R}^{3})$ and denote $\mathbb{P}_{1}:=\mathbb{R}P^{2t-1}=\mathbb{P}(\mathbb{R}^{2t-2}\times\mathbb{R}^{2}\times\{0\})$. Then
\begin{align*}
T_{1} = \{\gamma\in Y_{1}:\gamma(0)\in\mathbb{P}_{1},\ \gamma(1)\in
h^{\mathbb{P}_{1}}_{\gamma(0)} \mbox{ and } \mathrm{im}(\gamma)\subset
\ell_{\gamma(0)}^{+}\}.
\end{align*}

Yet another description of $T_{1}$ is
\begin{equation}
\label{eq:T1}T_{1}=\{(x_{0},x_{2})\in\mathbb{P}_{1}\times\mathbb{P}_{1}\, : \,
x_{2}\in h^{\mathbb{P}_{1}}_{x_{0}}\}.
\end{equation}
Indeed, given $(x_{0},x_{2})\in T_{1}$ there is a unique vertical half-circle
$C_{x_{0},x_{2}}$ contained in the upper half-sphere $\ell^{+}_{x_{0}}$ and
joining $x_{0}$ and $x_{2}$. The class $T$ is represented by the map
$(x_{0},x_{2})\mapsto C_{x_{0},x_{2}}$.

\smallskip

\noindent\emph{Description of $T^{2}$.} \smallskip

In order to describe a representing cycle for $T^{2}$ we first note that
$T_{1}$ is not self-transverse along the evaluation map at the endpoints. In
order to apply Lemma~\ref{lem:geom*} we use the following perturbation of
$T_{1}$. We denote $\mathbb{P}^{\prime}_{1}:=
\mathbb{P}(\mathbb{R}^{2t-2}\times\{0\}\times\mathbb{R}^{2})$ and identify $\mathbb{R}^{2t-2}
\times\{0\}\times\mathbb{R}^{2}$ with $\mathbb{C}^{t}$ and endow it with a
complex structure $J$ such that the subspace $\mathbb{R}^{2t-2}\times\{0\}$ is
$J$-invariant and inherits the same complex structure as from the above
identification $\mathbb{R}^{2t-2}\times\mathbb{R}^{2}\times\{0\}\equiv
\mathbb{C}^{t}$. We get a Hopf fibration $\mathbb{P}^{\prime}_{1}\to
\mathbb{C}P^{t-1}$ with fibers $h^{\mathbb{P}^{\prime}_{1}}_{x}$ and ``upper
hemispheres'' $\ell^{^{\prime}+}_{x}$, and we define
\[
T^{\prime}_{1}:=\{\gamma\in Y_{1}:\gamma(0)\in\mathbb{P}^{\prime}_{1},
\ \gamma(1)\in h^{\mathbb{P}^{\prime}_{1}}_{\gamma(0)} \mbox{ and }
\gamma\subset\ell_{\gamma(0)}^{^{\prime}+}\}.
\]
Since we can interpolate from $\mathbb{R}^{2}\times\{0\}$ to $\{0\}\times
\mathbb{R}^{2}$ inside $\mathbb{R}^{3}$ by a continuous family of planes
endowed with complex structures, we infer that $T^{\prime}_{1}$ is homologous
to $T_{1}$. Since the cycles $T^{\prime}_{1}$ and $T_{1}$ are transverse along
the evaluation maps at the endpoints, we obtain by Lemma~\ref{lem:geom*} that
$T^{2}$ is represented by the cycle
\begin{align*}
\Upsilon^{\prime}  &  := \{(\gamma,\delta)\in Y_{1}\times Y_{1} \, : \,
\gamma(0)\in\mathbb{P}_{1},\ \gamma(1)=\delta(0)\in 
h^{\mathbb{P}_{1}}_{\gamma(0)},\ \delta(1)\in h^{\mathbb{P}^{\prime}_{1}}_{\delta(0)},\\
&  \qquad\qquad\qquad\qquad\qquad\ \mathrm{im}(\gamma)\subset\ell_{\gamma
(0)}^{+} \mbox{ and } \mathrm{im}(\delta)\subset
\ell_{\delta(0)}^{^{\prime}+}\}.
\end{align*}
Note that $\mathbb{P}_{1}\cap\mathbb{P}^{\prime}_{1}=\mathbb{P}(\mathbb{R}
^{2t-2}\times\{0\}\times\mathbb{R}\times\{0\})=:\mathbb{P}_{2}$. Backwards
interpolation from $\mathbb{P}^{\prime}_{1}$ to $\mathbb{P}_{1}$ provides the
homologous cycle
\begin{align*}
\Upsilon &  := \!\!\{(\gamma,\delta)\in Y_{1}\times Y_{1}\, : \, \gamma
(0)\in\mathbb{P}_{1},\ \gamma(1)=\delta(0)\in h^{\mathbb{P}_{1}}_{\gamma(0)},
\ \delta(1)\in h^{\mathbb{P}_{1}}_{\gamma(0)}, \ \gamma(1)\in\mathbb{P}_{2},\\
&  \qquad\qquad\qquad\qquad\qquad\ \mathrm{im}(\gamma),\ \mathrm{im}(\delta)
\subset\ell_{\gamma(0)}^{+} \}.
\end{align*}

Yet another description of the cycle $\Upsilon$ that represents $T^{2}$ is the
following: it consists of triples $(x_{0},x_{1},x_{2})\in\mathbb{P}_{1}
\times\mathbb{P}_{1} \times\mathbb{P}_{1}$ such that $x_{0},x_{1},x_{2}$
belong to the same Hopf fiber $h^{\mathbb{P}_{1}}_{x_{0}}$ and $x_{1}
\in\mathbb{P}_{2}$. Indeed, such a triple determines uniquely a concatenated
path from the vertical half-circles connecting $x_{0}$ to $x_{1}$ and $x_{1}$
to $x_{2}$ inside $\ell^{+}_{x_{0}}$. This description also shows that
$\Upsilon$ is a manifold: the space of triples $(x_{0},x_{1},x_{2})$ which
belong to the same Hopf fiber is the double pull-back of the projectivized
Hopf fibration via the projection map. Imposing that $x_{1}$ belongs to
$\mathbb{P}_{2}$ amounts to taking the preimage of $\mathbb{P}_{2}$ via the
second projection, which is a submersion.

\smallskip

\noindent\emph{Proof that $T^{2}=T$.} \smallskip

The cycle $\Upsilon$ represents $T^{2}$ via the concatenation 
$C_{x_{0},x_{1}}\cdot C_{x_{1},x_{2}}$ of the two vertical half-circles $C_{x_{0},x_{1}}$ and
$C_{x_{1},x_{2}}$ that connect $x_{0}$ to $x_{1}$, respectively $x_{1}$ to
$x_{2}$ inside $\ell_{x_{0}}^{+}$. This concatenation can be homotoped to
$C_{x_{0},x_{2}}$ inside $\ell_{x_{0}}^{+}$ by concatenating the unique arc of
vertical half-circle running from $x_{0}$ to $C_{x_{1},x_{2}}(t)$ and the arc
of half-circle $C_{x_{1},x_{2}}|_{[t,1]}$, for $t\in[0,1]$. (See
Figure~\ref{fig:homotopyT2}. Note that this procedure is reminiscent of the
proof of Lemma~\ref{lem:S2=0}.)

\begin{figure}[ptb]
\begin{center}
\input{homotopyT2.pstex_t}
\end{center}
\caption{Deforming $\Upsilon$ to $T_{1}$.}
\label{fig:homotopyT2}
\end{figure}
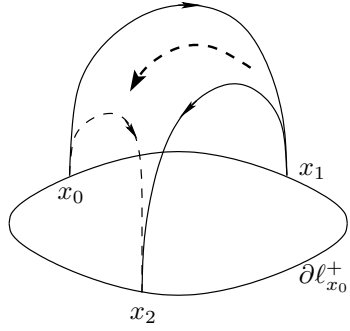

Thus $T^{2}$ is represented by the map
\[
C_{x_{0},x_{2}}:\Upsilon\to\mathcal{P}_{n}, \qquad(x_{0},x_{1},x_{2}
)\longmapsto C_{x_{0},x_{2}}.
\]

Adopting the description~\eqref{eq:T1} for the cycle $T_{1}$ that represents
the class $T$, we obtain a factorization
\begin{equation}
\label{eq:UpsilonT1}\xymatrix{ \Upsilon \ar[d]_{p_{02}} \ar[rr]^-{C_{x_0,x_2}} & & \mathcal{P}_n \\ T_1 \ar[urr]_-{C_{x_0,x_2}} }
\end{equation}
with $p_{02}:\Upsilon\to T_{1}$, $(x_{0},x_{1},x_{2})\longmapsto(x_{0},x_{2}
)$. The key point is that the map $p_{02}$ is smooth and generically
one-to-one. Indeed, the subspace $\mathbb{P}_{2}$ (on which the point $x_{1}$
is constrained to lie) contains entirely the Hopf fibers which lie in
$\mathbb{P}(\mathbb{R}^{2t-2}\times\{0\})$, but intersects in exactly
\emph{one} point the Hopf fibers which are not contained in $\mathbb{P}
(\mathbb{R}^{2t-2}\times\{0\})$.

Since the horizontal arrow in~\eqref{eq:UpsilonT1} represents $T^{2}$ and the
diagonal arrow represents $T$, we infer that $T^{2}=T$.
\end{proof}

\emph{This proves Theorem~\ref{thm:PCSproduct} in the case $n$ even, which was
the last remaining case.}

\begin{remark}
The reader is invited to compare the proofs of the identities $S^{2}=0$ and
$T^{2}=T$, respectively $HS+SH=U$ and $TH+HT=H$ in view of the discussion on
heredity in~\S \ref{sec:gen}(4).
\end{remark}


\bibliographystyle{abbrv}
\bibliography{../BIG}

\end{document}

%% file: geodesics.pstex_t
\begin{picture}(0,0)%
\includegraphics{geodesics.pstex}%
\end{picture}%
\setlength{\unitlength}{3947sp}%
\begingroup\makeatletter\ifx\SetFigFont\undefined%
\gdef\SetFigFont#1#2#3#4#5{%
  \reset@font\fontsize{#1}{#2pt}%
  \fontfamily{#3}\fontseries{#4}\fontshape{#5}%
  \selectfont}%
\fi\endgroup%
\begin{picture}(2730,2592)(661,-2020)
\put(676,-1711){\makebox(0,0)[lb]{\smash{{\SetFigFont{12}{14.4}{\rmdefault}{\mddefault}{\updefault}{\color[rgb]{0,0,0}$v^*$}%
}}}}
\put(3376, 14){\makebox(0,0)[lb]{\smash{{\SetFigFont{12}{14.4}{\rmdefault}{\mddefault}{\updefault}{\color[rgb]{0,0,0}$v$}%
}}}}
\put(1126,-1861){\makebox(0,0)[lb]{\smash{{\SetFigFont{12}{14.4}{\rmdefault}{\mddefault}{\updefault}{\color[rgb]{0,0,0}$\ell_{x,v}$}%
}}}}
\put(1576,-1192){\makebox(0,0)[lb]{\smash{{\SetFigFont{12}{14.4}{\rmdefault}{\mddefault}{\updefault}{\color[rgb]{0,0,0}$\mathbb{R}\ell_{x,v}$}%
}}}}
\put(3376,-886){\makebox(0,0)[lb]{\smash{{\SetFigFont{12}{14.4}{\rmdefault}{\mddefault}{\updefault}{\color[rgb]{0,0,0}$x$}%
}}}}
\put(2461,389){\makebox(0,0)[lb]{\smash{{\SetFigFont{12}{14.4}{\rmdefault}{\mddefault}{\updefault}{\color[rgb]{0,0,0}$\gamma_{x,v}$}%
}}}}
\put(2476,-361){\makebox(0,0)[lb]{\smash{{\SetFigFont{12}{14.4}{\rmdefault}{\mddefault}{\updefault}{\color[rgb]{0,0,0}$-Iv$}%
}}}}
\put(676,-886){\makebox(0,0)[lb]{\smash{{\SetFigFont{12}{14.4}{\rmdefault}{\mddefault}{\updefault}{\color[rgb]{0,0,0}$x^*$}%
}}}}
\end{picture}%

%% file: half-circles.pstex_t
\begin{picture}(0,0)%
\includegraphics{half-circles.pstex}%
\end{picture}%
\setlength{\unitlength}{3947sp}%
\begingroup\makeatletter\ifx\SetFigFont\undefined%
\gdef\SetFigFont#1#2#3#4#5{%
  \reset@font\fontsize{#1}{#2pt}%
  \fontfamily{#3}\fontseries{#4}\fontshape{#5}%
  \selectfont}%
\fi\endgroup%
\begin{picture}(2502,2667)(889,-2020)
\put(2101,-211){\makebox(0,0)[lb]{\smash{{\SetFigFont{12}{14.4}{\rmdefault}{\mddefault}{\updefault}{\color[rgb]{0,0,0}$C_{x,v,\theta}$}%
}}}}
\put(1351,-1186){\makebox(0,0)[lb]{\smash{{\SetFigFont{12}{14.4}{\rmdefault}{\mddefault}{\updefault}{\color[rgb]{0,0,0}$x'$}%
}}}}
\put(3376,-811){\makebox(0,0)[lb]{\smash{{\SetFigFont{12}{14.4}{\rmdefault}{\mddefault}{\updefault}{\color[rgb]{0,0,0}$x$}%
}}}}
\put(3376, 14){\makebox(0,0)[lb]{\smash{{\SetFigFont{12}{14.4}{\rmdefault}{\mddefault}{\updefault}{\color[rgb]{0,0,0}$v$}%
}}}}
\put(1951,464){\makebox(0,0)[lb]{\smash{{\SetFigFont{12}{14.4}{\rmdefault}{\mddefault}{\updefault}{\color[rgb]{0,0,0}$\ell_{x,v}$}%
}}}}
\put(2101,-1261){\makebox(0,0)[lb]{\smash{{\SetFigFont{12}{14.4}{\rmdefault}{\mddefault}{\updefault}{\color[rgb]{0,0,0}$\mathbb{R}\ell_{x,v}$}%
}}}}
\end{picture}%

%% file: Yk.pstex_t
\begin{picture}(0,0)%
\includegraphics{Yk.pstex}%
\end{picture}%
\setlength{\unitlength}{3947sp}%
\begingroup\makeatletter\ifx\SetFigFont\undefined%
\gdef\SetFigFont#1#2#3#4#5{%
  \reset@font\fontsize{#1}{#2pt}%
  \fontfamily{#3}\fontseries{#4}\fontshape{#5}%
  \selectfont}%
\fi\endgroup%
\begin{picture}(4694,1681)(729,-1520)
\put(4351,-586){\makebox(0,0)[lb]{\smash{{\SetFigFont{10}{12.0}{\rmdefault}{\mddefault}{\updefault}{\color[rgb]{0,0,0}$x_5$}%
}}}}
\put(1126, 14){\makebox(0,0)[lb]{\smash{{\SetFigFont{10}{12.0}{\rmdefault}{\mddefault}{\updefault}{\color[rgb]{0,0,0}$v_0$}%
}}}}
\put(5026,-886){\makebox(0,0)[lb]{\smash{{\SetFigFont{10}{12.0}{\rmdefault}{\mddefault}{\updefault}{\color[rgb]{0,0,0}$\mathbb{R}P^n$}%
}}}}
\put(2626,-586){\makebox(0,0)[lb]{\smash{{\SetFigFont{10}{12.0}{\rmdefault}{\mddefault}{\updefault}{\color[rgb]{0,0,0}$x_1$}%
}}}}
\put(3151,-811){\makebox(0,0)[lb]{\smash{{\SetFigFont{10}{12.0}{\rmdefault}{\mddefault}{\updefault}{\color[rgb]{0,0,0}$x_2$}%
}}}}
\put(2176,-811){\makebox(0,0)[lb]{\smash{{\SetFigFont{10}{12.0}{\rmdefault}{\mddefault}{\updefault}{\color[rgb]{0,0,0}$x_3$}%
}}}}
\put(1201,-811){\makebox(0,0)[lb]{\smash{{\SetFigFont{10}{12.0}{\rmdefault}{\mddefault}{\updefault}{\color[rgb]{0,0,0}$x_0$}%
}}}}
\put(3751,-586){\makebox(0,0)[lb]{\smash{{\SetFigFont{10}{12.0}{\rmdefault}{\mddefault}{\updefault}{\color[rgb]{0,0,0}$x_4$}%
}}}}
\end{picture}%

%% file: IJ.pstex_t
\begin{picture}(0,0)%
\includegraphics{IJ.pstex}%
\end{picture}%
\setlength{\unitlength}{3947sp}%
\begingroup\makeatletter\ifx\SetFigFont\undefined%
\gdef\SetFigFont#1#2#3#4#5{%
  \reset@font\fontsize{#1}{#2pt}%
  \fontfamily{#3}\fontseries{#4}\fontshape{#5}%
  \selectfont}%
\fi\endgroup%
\begin{picture}(2778,2592)(613,-2020)
\put(2701,-361){\makebox(0,0)[lb]{\smash{{\SetFigFont{12}{14.4}{\rmdefault}{\mddefault}{\updefault}{\color[rgb]{0,0,0}$Jx$}%
}}}}
\put(3376, 14){\makebox(0,0)[lb]{\smash{{\SetFigFont{12}{14.4}{\rmdefault}{\mddefault}{\updefault}{\color[rgb]{0,0,0}$IJx$}%
}}}}
\put(1126,-1861){\makebox(0,0)[lb]{\smash{{\SetFigFont{12}{14.4}{\rmdefault}{\mddefault}{\updefault}{\color[rgb]{0,0,0}$\ell_{x,v}$}%
}}}}
\put(628,-886){\makebox(0,0)[lb]{\smash{{\SetFigFont{12}{14.4}{\rmdefault}{\mddefault}{\updefault}{\color[rgb]{0,0,0}$-x$}%
}}}}
\put(3376,-886){\makebox(0,0)[lb]{\smash{{\SetFigFont{12}{14.4}{\rmdefault}{\mddefault}{\updefault}{\color[rgb]{0,0,0}$x$}%
}}}}
\put(2461,389){\makebox(0,0)[lb]{\smash{{\SetFigFont{12}{14.4}{\rmdefault}{\mddefault}{\updefault}{\color[rgb]{0,0,0}$\gamma_{x,v}$}%
}}}}
\put(676, 89){\makebox(0,0)[lb]{\smash{{\SetFigFont{12}{14.4}{\rmdefault}{\mddefault}{\updefault}{\color[rgb]{0,0,0}$-IJx$}%
}}}}
\put(1201,-1411){\makebox(0,0)[lb]{\smash{{\SetFigFont{12}{14.4}{\rmdefault}{\mddefault}{\updefault}{\color[rgb]{0,0,0}$J(-x)$}%
}}}}
\put(1351,-361){\makebox(0,0)[lb]{\smash{{\SetFigFont{12}{14.4}{\rmdefault}{\mddefault}{\updefault}{\color[rgb]{0,0,0}$\ell^+_x$}%
}}}}
\put(1951,-1261){\makebox(0,0)[lb]{\smash{{\SetFigFont{12}{14.4}{\rmdefault}{\mddefault}{\updefault}{\color[rgb]{0,0,0}$\mathbb{R}\ell_{x,v}$}%
}}}}
\end{picture}%

%% file: torus.pstex_t
\begin{picture}(0,0)%
\includegraphics{torus.pstex}%
\end{picture}%
\setlength{\unitlength}{3947sp}%
\begingroup\makeatletter\ifx\SetFigFont\undefined%
\gdef\SetFigFont#1#2#3#4#5{%
  \reset@font\fontsize{#1}{#2pt}%
  \fontfamily{#3}\fontseries{#4}\fontshape{#5}%
  \selectfont}%
\fi\endgroup%
\begin{picture}(2373,1954)(661,-1832)
\put(2476,-286){\makebox(0,0)[lb]{\smash{{\SetFigFont{8}{9.6}{\rmdefault}{\mddefault}{\updefault}{\color[rgb]{0,0,0}$U$}%
}}}}
\put(1426,-286){\makebox(0,0)[lb]{\smash{{\SetFigFont{8}{9.6}{\rmdefault}{\mddefault}{\updefault}{\color[rgb]{0,0,0}$SP$}%
}}}}
\put(676,-811){\makebox(0,0)[lb]{\smash{{\SetFigFont{8}{9.6}{\rmdefault}{\mddefault}{\updefault}{\color[rgb]{0,0,0}$\{*\}\times\mathbb{R}P^1$}%
}}}}
\put(2476,-1486){\makebox(0,0)[lb]{\smash{{\SetFigFont{8}{9.6}{\rmdefault}{\mddefault}{\updefault}{\color[rgb]{0,0,0}$PS$}%
}}}}
\put(1876,-1786){\makebox(0,0)[lb]{\smash{{\SetFigFont{8}{9.6}{\rmdefault}{\mddefault}{\updefault}{\color[rgb]{0,0,0}$\mathbb{R}P^1\times\{*\}$}%
}}}}
\end{picture}%

%% file: homotopyT2.pstex_t
\begin{picture}(0,0)%
\includegraphics{homotopyT2.pstex}%
\end{picture}%
\setlength{\unitlength}{3947sp}%
\begingroup\makeatletter\ifx\SetFigFont\undefined%
\gdef\SetFigFont#1#2#3#4#5{%
  \reset@font\fontsize{#1}{#2pt}%
  \fontfamily{#3}\fontseries{#4}\fontshape{#5}%
  \selectfont}%
\fi\endgroup%
\begin{picture}(2124,2058)(589,-725)
\put(2401,239){\makebox(0,0)[lb]{\smash{{\SetFigFont{10}{12.0}{\rmdefault}{\mddefault}{\updefault}{\color[rgb]{0,0,0}$x_1$}%
}}}}
\put(901, 89){\makebox(0,0)[lb]{\smash{{\SetFigFont{10}{12.0}{\rmdefault}{\mddefault}{\updefault}{\color[rgb]{0,0,0}$x_0$}%
}}}}
\put(1351,-661){\makebox(0,0)[lb]{\smash{{\SetFigFont{10}{12.0}{\rmdefault}{\mddefault}{\updefault}{\color[rgb]{0,0,0}$x_2$}%
}}}}
\put(2401,-436){\makebox(0,0)[lb]{\smash{{\SetFigFont{10}{12.0}{\rmdefault}{\mddefault}{\updefault}{\color[rgb]{0,0,0}$\partial\ell^+_{x_0}$}%
}}}}
\end{picture}%

%% file: Paths_2013-11-27-arxiv.bbl
\def\cprime{$'$} \def\cprime{$'$}
\begin{thebibliography}{10}

\bibitem{Abbondandolo-Portaluri-Schwarz}
A.~Abbondandolo, A.~Portaluri, and M.~Schwarz.
\newblock The homology of path spaces and {F}loer homology with conormal
  boundary conditions.
\newblock {\em J. Fixed Point Theory Appl.}, 4(2):263--293, 2008.

\bibitem{AS2}
A.~Abbondandolo and M.~Schwarz.
\newblock Floer homology of cotangent bundles and the loop product.
\newblock {\em Geom. Topol.}, 14(3):1569--1722, 2010.

\bibitem{Abouzaid2011a}
M.~Abouzaid.
\newblock A cotangent fibre generates the {F}ukaya category.
\newblock {\em Adv. Math.}, 228(2):894--939, 2011.

\bibitem{Abouzaid-Seidel}
M.~Abouzaid and P.~Seidel.
\newblock An open string analogue of {V}iterbo functoriality.
\newblock {\em Geom. Topol.}, 14(2):627--718, 2010.

\bibitem{Atiyah-Bott}
M.~F. Atiyah and R.~Bott.
\newblock The {Y}ang-{M}ills equations over {R}iemann surfaces.
\newblock {\em Philos. Trans. Roy. Soc. London Ser. A}, 308(1505):523--615,
  1983.

\bibitem{Besse}
A.~L. Besse.
\newblock {\em Manifolds all of whose geodesics are closed}, volume~93 of {\em
  Ergebnisse der Mathematik und ihrer Grenzgebiete [Results in Mathematics and
  Related Areas]}.
\newblock Springer-Verlag, Berlin, 1978.
\newblock With appendices by D. B. A. Epstein, J.-P. Bourguignon, L.
  B{\'e}rard-Bergery, M. Berger and J. L. Kazdan.

\bibitem{Blumberg-Cohen-Teleman}
A.~J. Blumberg, R.~L. Cohen, and C.~Teleman.
\newblock Open-closed field theories, string topology, and {H}ochschild
  homology.
\newblock In {\em Alpine perspectives on algebraic topology}, volume 504 of
  {\em Contemp. Math.}, pages 53--76. Amer. Math. Soc., Providence, RI, 2009.

\bibitem{Bott-nondegenerate}
R.~Bott.
\newblock Nondegenerate critical manifolds.
\newblock {\em Ann. of Math. (2)}, 60:248--261, 1954.

\bibitem{Bott-Lectures_Morse_theory}
R.~Bott.
\newblock {\em Morse theory and its application to homotopy theory; lectures}.
\newblock Bonn Math. Inst. Vorlesungsarbeiten. Mathematisches Institut der
  Universit\"at Bonn, 1960.

\bibitem{Bott-Samelson}
R.~Bott and H.~Samelson.
\newblock Applications of the theory of {M}orse to symmetric spaces.
\newblock {\em Amer. J. Math.}, 80:964--1029, 1958.

\bibitem{Chataur}
D.~Chataur.
\newblock A bordism approach to string topology.
\newblock {\em Int. Math. Res. Not.}, (46):2829--2875, 2005.

\bibitem{Cohen-Jones}
R.~L. Cohen and J.~D.~S. Jones.
\newblock A homotopy theoretic realization of string topology.
\newblock {\em Math. Ann.}, 324(4):773--798, 2002.

\bibitem{Cohen-Jones-Yan}
R.~L. Cohen, J.~D.~S. Jones, and J.~Yan.
\newblock The loop homology algebra of spheres and projective spaces.
\newblock In {\em Categorical decomposition techniques in algebraic topology
  ({I}sle of {S}kye, 2001)}, volume 215 of {\em Progr. Math.}, pages 77--92.
  Birkh\"auser, Basel, 2004.

\bibitem{Gallot-Hulin-Lafontaine}
S.~Gallot, D.~Hulin, and J.~Lafontaine.
\newblock {\em Riemannian geometry}.
\newblock Universitext. Springer-Verlag, Berlin, third edition, 2004.

\bibitem{Goresky-Hingston}
M.~Goresky and N.~Hingston.
\newblock Loop products and closed geodesics.
\newblock {\em Duke Math. J.}, 150(1):117--209, 2009.

\bibitem{Greenberg-Harper}
M.~J. Greenberg and J.~R. Harper.
\newblock {\em Algebraic topology}, volume~58 of {\em Mathematics Lecture Note
  Series}.
\newblock Benjamin/Cummings Publishing Co. Inc. Advanced Book Program, Reading,
  Mass., 1981.
\newblock A first course.

\bibitem{Hingston}
N.~Hingston.
\newblock Equivariant {M}orse theory and closed geodesics.
\newblock {\em J. Differential Geom.}, 19(1):85--116, 1984.

\bibitem{Hingston-Kalish}
N.~Hingston and D.~Kalish.
\newblock The {M}orse index theorem in the degenerate endmanifold case.
\newblock {\em Proc. Amer. Math. Soc.}, 118(2):663--668, 1993.

\bibitem{Klingenberg}
W.~Klingenberg.
\newblock {\em Lectures on closed geodesics}.
\newblock Springer-Verlag, Berlin, 1978.
\newblock Grundlehren der Mathematischen Wissenschaften, Vol. 230.

\bibitem{Milnor-Morse_theory}
J.~Milnor.
\newblock {\em Morse theory}.
\newblock Based on lecture notes by M. Spivak and R. Wells. Annals of
  Mathematics Studies, No. 51. Princeton University Press, Princeton, N.J.,
  1963.

\bibitem{Morse}
M.~Morse.
\newblock {\em The calculus of variations in the large}, volume~18 of {\em
  American Mathematical Society Colloquium Publications}.
\newblock American Mathematical Society, Providence, RI, 1996.
\newblock Reprint of the 1932 original.

\bibitem{Oancea-Montreal}
A.~Oancea.
\newblock String topology and the decorated moduli space.
\newblock Talk at the workshop on ``{J}-holomorphic {C}urves in {S}ymplectic
  {G}eometry, {T}opology and {D}ynamics", Montreal, May 10, 2013.

\bibitem{Ziller1977}
W.~Ziller.
\newblock The free loop space of globally symmetric spaces.
\newblock {\em Invent. Math.}, 41(1):1--22, 1977.

\end{thebibliography}
